\documentclass[preprint,12pt]{elsarticle}
\usepackage{psfrag}
\usepackage{a4wide}
\usepackage{rotating}
\usepackage{color}
\usepackage{amssymb}
\usepackage{amsmath}
\usepackage{amsthm}
\usepackage{verbatim}
\usepackage{subfigure}
\usepackage{tabularx}
\usepackage{tikz}
\usepackage{pgfplots}
\usepgfplotslibrary{colorbrewer}
\usepackage{soul}

\def\bfm#1{\boldsymbol{#1}}
\newcommand{\bb}[1]{\bfm{#1}}
\newcommand{\f}[1]{\mathbf{#1}}
\newcommand{\R}{\mathbb R}
\newcommand{\N}{\mathbb N}
\newcommand{\UT}{T}
\newcommand{\US}[2]{\mathbb{S}_{#1}^{#2}}
\newcommand{\UXI}[2]{\Xi_{#1}^{#2}}
\newcommand{\UN}[3]{N_{#3,#1}^{#2}}
\newcommand{\UV}{\mathbb{V}}
\newcommand{\UW}{\mathbb{W}}
\newcommand{\PhiB}{\Phi}
\newcommand{\phiS}[3]{\phi_{#2,#3}^{\Omega^{(#1)}}}
\newcommand{\phiT}[2]{\phi_{#2}^{\Gamma_{#1}}}

\newcommand{\coefa}[2]{\widehat{b}^{#1}_{#2}}
\newcommand{\coefb}[2]{\widetilde{b}\mbox{}^{#1}_{#2}}
\newcommand{\coefc}[2]{\overline{b}\mbox{}^{#1}_{#2}}
\newcommand{\tauN}[1]{\tau_{#1}}
\newcommand{\I}{\mathbf{I}}
\newcommand{\J}[2]{\mathbf{J}_{{#1},{#2}}}
\newcommand{\zetaN}[1]{\zeta_{#1}}

\newcommand{\supp}{\text{supp\,}}
\newcommand{\myspan}{\text{span\,}}
\renewcommand{\vec}[1]{\bb{#1}}

\newcommand{\HUN}[4]{N_{#3,#1}^{#2,#4}}

\newcommand{\HhatPhiT}[3]{\widehat{\vec{\phi}}^{(#1),#3}_{\Gamma_{#2}}}
\newcommand{\HVTN}[2]{{\bf N}_{#1}^{#2}}

\newcommand{\hatPhiS}[2]{\widehat{\vec{\phi}}_{\Omega^{(#2)}}^{(#1)}}
\newcommand{\hatPhiT}[2]{\mathbf{\widehat{\vec{\phi}}}^{(#1)}_{\Gamma_{#2}}}
\newcommand{\VTN}[1]{{\bf N}_{#1}}

\newcommand{\convergenceslope}[6] 
{
\draw [color=black,mark=none,#6]
(axis cs: #1, #3) -- (axis cs: #2, 
{exp((#4*(-1)) * ln((#2)/(#1)) + ln(#3))}) 
-- node [midway, #5, below]{1} (axis cs: #1, 
{exp((#4*(-1)) * ln((#2)/(#1)) + ln(#3))}) 
-- node [midway, #5, left]{${#4}$} (axis cs: #1, #3);
}

\newtheorem{thm}{Theorem}

\newtheorem{prop}{Proposition}

\theoremstyle{definition}
\newtheorem{ex}{Example}
\newtheorem{defn}{Definition}
\newtheorem{rem}{Remark}

\newproof{pf}{proof}

\advance\textheight by 0.4cm
\advance\topmargin by -0.2cm

\begin{document}

\begin{frontmatter}

\title{Isogeometric analysis with $C^1$ hierarchical functions on planar two-patch geometries}


\author[flo]{Cesare Bracco}
\ead{cesare.bracco@unifi.it}

\author[flo]{Carlotta Giannelli}
\ead{carlotta.giannelli@unifi.it}

\author[lnz]{Mario Kapl\corref{cor}}
\ead{mario.kapl@ricam.oeaw.ac.at}

\author[lau,pav]{Rafael V\'{a}zquez}
\ead{rafael.vazquez@epfl.ch}
 
\address[flo]{Dipartimento di Matematica e Informatica ``U. Dini'', \\Universit\`{a} degli Studi di Firenze, Florence, Italy}
\address[lnz]{Johann Radon Institute for Computational and Applied Mathematics, \\Austrian Academy of Sciences, Linz, Austria}
\address[lau]{Institute of Mathematics, \'{E}cole Polytechnique F\'{e}d\'{e}rale de Lausanne, Lausanne, Switzerland}
\address[pav]{Istituto di Matematica Applicata e Tecnologie Informatiche `E.~Magenes' del CNR, Pavia, Italy}

\cortext[cor]{Corresponding author}
   
\begin{abstract}
Adaptive isogeometric methods for the solution of partial differential equations rely on the construction of locally refinable spline spaces. A simple and efficient 
way to obtain these spaces is to apply the multi-level construction of hierarchical splines, that can be used on single-patch domains or in multi-patch domains with $C^0$ continuity 
across the patch interfaces. Due to the benefits of higher continuity in isogeometric methods, recent works investigated the construction of spline spaces with global $C^1$ continuity 
on two or more patches. In this paper, we show how these approaches can be combined with the hierarchical construction to obtain global $C^1$ continuous hierarchical splines on 
two-patch domains. A selection of numerical examples is presented to highlight the features and effectivity of the construction.
\end{abstract}

\begin{keyword}
Isogeometric analysis \sep Geometric continuity \sep Two-patch domain \sep Hierarchical splines \sep Local refinement
\MSC 65D07\sep 65D17\sep 65N30
\end{keyword}

\end{frontmatter}

\section{Introduction}  \label{sec:into}

Isogeometric Analysis (IgA) is a framework for numerically solving partial differential equations (PDEs), see \cite{ANU:9260759,CottrellBook,HuCoBa04}, by using the same (spline) 
function space for describing the geometry (i.e. the computational domain) and for representing the solution of the considered PDE. One of the strong points of IgA compared to finite 
elements is the possibility to easily construct $C^1$ spline spaces, 
and to use them for solving 
fourth order PDEs by applying a Galerkin discretization to their variational formulation. Examples of fourth order problems with practical relevance (in the frame of IgA) are e.g. 
the biharmonic equation~\cite{CoSaTa16,KaBuBeJu16,TaDe14}, the Kirchhoff-Love 
shells~\cite{ABBLRS-stream,benson2011large,kiendl-bazilevs-hsu-wuechner-bletzinger-10,kiendl-bletzinger-linhard-09} 
and the Cahn-Hilliard equation~\cite{gomez2008isogeometric,GoCaHu09,LiDeEvBoHu13}.

Adaptive isogeometric methods can be developed by combining the IgA framework with spline spaces that have local refinement capabilities. 
Hierarchical B-splines \cite{kraft1997,vuong2011} and truncated hierarchical B-splines \cite{giannelli2012,giannelli2014} are probably the adaptive spline technologies that have 
been studied more in detail in the adaptive IgA framework \cite{buffa2016c,buffa2017b,gantner2017}. Their multi-level structure makes them easy to implement, with the evaluation of 
basis functions obtained via a recursive use of two-level relation due to nestedness of levels \cite{dangella2018,garau2018,hennig2016}. Hierarchical B-splines have been successfully 
applied for the adaptive discretization of fourth order PDEs, and in particular for phase-field models used in the simulation of brittle fracture \cite{hennig2018,hennig2016} or tumor 
growth \cite{lorenzo2017}.

While the construction of $C^1$ spaces is trivial in a single-patch domain, either using B-splines or hierarchical B-splines, the same is not true for general multi-patch domains. 
The construction of $C^1$ spline spaces over multi-patch domains is based on the concept of geometric continuity~\cite{HoLa93,Pe02}, which 
is a well-known framework in computer-aided design (CAD) for the design of smooth multi-patch surfaces. The core idea is to employ the fact that an isogeometric function is 
$C^1$-smooth if and only if the associated multi-patch graph surface is $G^1$-smooth~\cite{Pe15}, i.e., it is geometrically continuous of order~$1$. 

In the last few years there has been an increasing effort to provide methods for the construction of $C^1$ isogeometric spline spaces over general multi-patch domains.
The existing methods for planar domains
can be roughly classified into two groups depending on the used parameterization for the multi-patch domain. 
The first approach relies on a multi-patch parameterization which is $C^1$-smooth everywhere except in the neighborhood of extraordinary vertices (i.e. vertices with 
valencies different to four), where the parameterization is singular, see e.g.~\cite{NgPe16,ToSpHu17b,ToSpHu17}, or consists of a special construction, 
see e.g.~\cite{Pe15-2,KaPe18,NgKaPe15}. The methods \cite{NgPe16,ToSpHu17b,ToSpHu17} use a singular parameterization with patches in the vicinity of an 
extraordinary vertex, which belong to a specific class of degenerate (B\'{e}zier) patches introduced in \cite{Re97}, and that allow, despite having singularities, the design of 
globally $C^1$ isogeometric spaces. The techniques~\cite{Pe15-2,KaPe18,NgKaPe15} are based on $G^1$ multi-patch surface constructions, where the obtained surface in the neighborhood 
of an extraordinary vertex consists of patches of slightly higher degree~\cite{Pe15-2,NgKaPe15} and is generated by means of a particular subdivision scheme~\cite{KaPe18}. 
As a special case of the first approach can be seen the constructions in \cite{Peters2,ToSpHiHu16}, that employ a polar framework to generate $C^1$ spline spaces.

The second approach, on which we will focus, uses a particular class of regular $C^0$ multi-patch parameterizations, called analysis-suitable $G^1$ multi-patch 
parameterization~\cite{CoSaTa16}. 
The class of analysis-suitable $G^1$ multi-patch geometries characterizes the regular $C^0$ multi-patch parameterizations that allow the design of $C^1$ isogeometric spline 
spaces with optimal approximation properties, see~\cite{CoSaTa16,KaSaTa17b}, and includes for instance the subclass of bilinear multi-patch 
parameterizations~\cite{BeMa14,KaBuBeJu16,KaViJu15}. An algorithm for the construction of analysis-suitable $G^1$ parameterizations for complex multi-patch domains was presented 
in~\cite{KaSaTa17b}. The main idea of this approach is to analyze the entire space of $C^1$ isogeometric functions over the given multi-patch geometry to generate a basis of 
this space or of a suitable subspace. 
While the methods in \cite{BeMa14,KaBuBeJu16,KaViJu15} are mainly restricted to (mapped) bilinear multi-patch parameterizations, the 
techniques~\cite{BlMoVi17,KaSaTa17a,KaSaTa19b,KaSaTa17c,mourrain2015geometrically} can also deal with more general multi-patch geometries. An alternative but related approach 
comprises the constructions~\cite{ChAnRa18,ChAnRa19} for general $C^0$ multi-patch parameterizations, which increase the degree 
of the constructed spline functions in the neighborhood of the common interfaces to obtain $C^1$ isogeometric spaces with good approximation properties. 

In this work, we extend for the case of two-patch domains the second approach from above to the construction of \emph{hierarchical} $C^1$ isogeometric spaces on 
analysis-suitable $G^1$ geometries, using the abstract framework for the definition of hierarchical splines detailed in \cite{giannelli2014}. 
We show that the basis functions of the considered $C^1$ space on analysis-suitable $G^1$ two-patch parameterizations, which is a subspace of the space~\cite{KaSaTa17a} 
inspired by~\cite{KaSaTa17c}, satisfy the required properties given in \cite{giannelli2014}, and in particular that the basis functions are locally linearly independent 
(see Section~\ref{sec:abstract} for details). Note that in case of a multi-patch domain, the general framework for the construction of hierarchical splines~\cite{giannelli2014} cannot 
be used anymore, since the appropriate $C^1$ basis functions~\cite{KaSaTa17c} can be locally linearly dependent. Therefore, the development of another 
approach as~\cite{giannelli2014} would be needed for the multi-patch case, which is beyond the scope of this paper.

For the construction of the hierarchical $C^1$ spline spaces on analysis-suitable $G^1$ two-patch geometries, we also explore the explicit expression for the relation between 
$C^1$ basis functions of two consecutive levels, expressing coarse basis functions as linear combinations of fine basis functions. This relation is exploited for the implementation 
of hierarchical splines as in \cite{garau2018,hennig2016}. A series of numerical tests are presented, that are run with the help of the Matlab/Octave 
code GeoPDEs \cite{garau2018,vazquez2016}. 

The remainder of the paper is organized as follows. Section~\ref{sec:C1space} recalls the concept of analysis-suitable $G^1$ two-patch geometries and 
presents the used $C^1$ isogeometric spline space over this class of parameterizations. In Section~\ref{sec:C1hierarchical}, we develop the (theoretical) 
framework to employ this space to construct $C^1$ hierarchical isogeometric spline spaces, which 
includes the verification of the nested nature of this kind of spaces, as well as the proof of the local linear independence of the one-level basis functions. 
Additional details of the $C^1$ hierarchical construction, such as the refinement masks of the basis functions for the different levels, are discussed in 
Section~\ref{sec:implementation} with focus on implementation aspects. The generated hierarchical spaces are then used in Section~\ref{sec:examples} to numerically solve the 
laplacian and bilaplacian equations on two-patch geometries, where the numerical results demonstrate the potential of our $C^1$ hierarchical construction for applications in IgA. 
Finally, the concluding remarks can be found in Section~\ref{sec:conclusions}. The construction of the non-trivial analysis-suitable $G^1$ two-patch parameterization used in some of 
the numerical examples is described in detail in \ref{app:geometry}. For easiness of reading, we include at the end of the paper a list of symbols with the main notation 
used in this work.

\section{$C^1$ isogeometric spaces on two-patch geometries} \label{sec:C1space}

In this section, we introduce the specific class of two-patch geometries and the $C^1$ isogeometric spaces which will be used throughout the paper.

\subsection{Analysis-suitable $G^1$ two-patch geometries} \label{subsec:analysis_two-patch}

We present a particular class of planar two-patch geometries, called analysis-suitable $G^1$ two-patch geometries, which was introduced in~\cite{CoSaTa16}. This class is of 
importance since it comprises exactly those two-patch geometries which are suitable for the construction of $C^1$ isogeometric spaces with optimal approximation 
properties, see~\cite{CoSaTa16,KaSaTa17b}. The most prominent member is the subclass of bilinear two-patch parameterizations, but it was demonstrated in~\cite{KaSaTa17b} 
that the class is much wider and allows the design of generic planar two-patch domains.

Let $k,p,r \in \N$ with degree $p \geq 3$ and regularity $1 \leq r \leq p-2$. Let us also introduce the ordered set of internal breakpoints 
$\UT = \{ \tauN{1}, \tauN{2}, \ldots, \tauN{k}\}$, with $0 < \tau_{i} < \tau_{i+1} < 1$ for all $1 \leq i \leq k$. We denote by $\US{p}{r}$ the univariate spline space in $[0,1]$ with 
respect to the open knot vector 
\begin{equation}\label{eq:kv}
 \UXI{p}{r} = \{\underbrace{0,\ldots,0}_{(p+1)-\mbox{\scriptsize times}},
\underbrace{\tauN{1},\ldots ,\tauN{1}}_{(p-r) - \mbox{\scriptsize times}}, 
\underbrace{ \tauN{2},\ldots ,\tauN{2}}_{(p-r) - \mbox{\scriptsize times}},\ldots, 
\underbrace{ \tauN{k},\ldots ,\tauN{k}}_{(p-r) - \mbox{\scriptsize times}},
\underbrace{1,\ldots,1}_{(p+1)-\mbox{\scriptsize times}}\},
\end{equation}
and let $\UN{p}{r}{i}$, $i \in \I = \{ 0,\ldots,p+k(p-r) \}$, 
be the associated B-splines. Note that the parameter $r$ specifies 
the resulting $C^r$-continuity of the spline space~$\US{p}{r}$. 
We will also make use of the subspaces of higher regularity and lower degree, respectively $\US{p}{r+1}$ and $\US{p-1}{r}$, defined from the same internal breakpoints, 
and we will use an analogous 
notation for their basis functions. Furthermore, we denote by $n$, $n_0$ and $n_1$ the dimensions of the spline spaces $\US{p}{r}$, $\US{p}{r+1}$ and 
$\US{p-1}{r}$, respectively, which are given by
\[
 n= p+1+ k(p-r), \mbox{ }n_0= p+1+ k(p-r-1) \mbox{ and }n_1= p + k(p-r-1),
\]
and, analogously to $\I$, we introduce the index sets 
\[
 \I_0 = \{0, \ldots, n_{0}-1  \}, \qquad
 \I_1=\{0,\ldots, n_{1}-1 \},
\]
corresponding to basis functions in $\US{p}{r+1}$ and $\US{p-1}{r}$, respectively.

Let $\f{F}^{(L)}, \f{F}^{(R)} \in (\US{p}{r} \otimes \US{p}{r})^2$ be two regular spline parameterizations, 
whose images $\f{F}^{(L)}([0,1]^2)$ and $\f{F}^{(R)}([0,1]^2)$ define the two quadrilateral patches $\Omega^{(L)}$ and 
$\Omega^{(R)}$ via $\f{F}^{(S)}([0,1]^2) = \Omega^{(S)}$, $S \in \{L,R \}$. The regular, bijective mapping 
$\f{F}^{(S)} : [0,1]^2 \rightarrow \Omega^{(S)}$, $S \in \{L,R \}$, is called \emph{geometry mapping}, and possesses a spline representation
\[
\f{F}^{(S)}(\xi_1,\xi_2)= \sum_{i \in \I} \sum_{j \in \I}\f{c}^{(S)}_{i,j} \UN{p}{r}{i}(\xi_1) \UN{p}{r}{j}(\xi_2), \quad \f{c}^{(S)}_{i,j} \in \R^{2}.
\] 
We assume that the two patches $\Omega^{(L)}$ and $\Omega^{(R)}$ form a planar two-patch domain $\Omega=\Omega^{(L)} \cup \Omega^{(R)}$, which share one whole edge as 
common interface $\Gamma = \Omega^{(L)} \cap \Omega^{(R)}$. In addition, and without loss of generality, we assume that the common interface~$\Gamma$ is parameterized by 
$\f{F}_{0} : [0,1] \to \Gamma$ via
\[
\f{F}_{0}(\xi_2) =  \f{F}^{(L)}(0,\xi_2) = \f{F}^{(R)}(0,\xi_2), \quad \xi_2 \in [0,1], 
\]
and denote by $\f{F}$ the two-patch parameterization (also called \emph{two-patch geometry}) consisting of the two spline parameterizations $\f{F}^{(L)}$ 
and $\f{F}^{(R)}$.

\begin{rem}
For simplicity, we have restricted ourselves to a univariate spline space~$\US{p}{r}$ with the same knot multiplicity for all inner knots. Instead, a univariate 
spline space with different inner knot multiplicities can be used, as long as the multiplicity of each inner knot is at least $2$ and at most $p-1$.
Note that the subspaces~$\US{p}{r+1}$ and $\US{p-1}{r}$ should also be replaced by suitable spline spaces of regularity increased by one at each inner knot, and degree reduced by one, 
respectively. Furthermore, it is also possible to use different univariate spline spaces for both Cartesian directions and for both geometry mappings, with the requirement that both 
patches must have the same univariate spline space in $\xi_2$-direction.
\end{rem}

The two geometry mappings~$\f{F}^{(L)}$ and $\f{F}^{(R)}$ uniquely determine up to a common function~$\gamma : [0,1] \rightarrow \R$ (with $\gamma \neq 0$), the functions 
$\alpha^{(L)}$, $\alpha^{(R)}$, $\beta: [0,1] \rightarrow \R$ given by
\[
\alpha^{(S)}(\xi_2) = \gamma(\xi_2) \det \left( \partial_1 \f{F}^{(S)}(0,\xi_2), \partial_2 \f{F}^{(S)}(0,\xi_2)  \right), \quad S \in \{L,R \} ,
\]
and 
\[
\beta(\xi_2) = \gamma(\xi_2) \det \left( \partial_1 \f{F}^{(L)}(0,\xi_2), \partial_1 \f{F}^{(R)}(0,\xi_2) \right),
\]
satisfying for $\xi_2 \in [0,1 ]$
\begin{equation} \label{eq:alpha_cond}
 \alpha^ {(L)}(\xi_2) \alpha^{(R)}(\xi_2) < 0
\end{equation}
and
\begin{equation} \label{eq:geometry_cond}
 \alpha^{(R)} \partial_1 \f{F}^{(L)}(0,\xi_2) - \alpha^{(L)}(\xi_2) \partial_1 \f{F}^{(R)}(0,\xi_2) + \beta(\xi_2) \partial_2 \f{F}^{(L)}(0,\xi_2) = \f{0}.
\end{equation}
In addition, there exist non-unique functions~$\beta^{(L)}$ and $\beta^{(R)}:[0,1] \rightarrow \R$ such that
\begin{equation} \label{eq:beta_cond}
\beta(\xi_2) = \alpha^{(L)}(\xi_2) \beta^{(R)}(\xi_2) - \alpha^{(R)}(\xi_2) \beta^{(L)}(\xi_2), 
\end{equation}
see e.g. \cite{CoSaTa16,Pe02}. The two-patch geometry~$\f{F}$ is called \emph{analysis-suitable $G^1$} if there exist linear functions~$\alpha^{(S)},\beta^{(S)}$, 
$S \in \{L,R \}$ with $\alpha^{(L)}$ and $\alpha^{(R)}$ relatively prime\footnote{Two polynomials are relatively prime if their greatest common divisor has degree zero.} 
such that equations~\eqref{eq:alpha_cond}-\eqref{eq:beta_cond} are satisfied 
for $\xi_2 \in [0,1 ]$, see~\cite{CoSaTa16,KaSaTa17a}. Note that requiring that $\alpha^{(L)}$ and $\alpha^{(R)}$ are relatively prime is not restrictive: if $\alpha^{(L)}$ and 
$\alpha^{(R)}$ share a common factor, it is a factor of $\gamma$ too, thus $\alpha^{(L)}$ and $\alpha^{(R)}$ can be made relatively prime by dividing by such a factor. 

In the following, we will only consider planar two-patch domains~$\Omega$ which are described by analysis-suitable $G^1$ two-patch geometries~$\f{F}$. Furthermore, we 
select those linear functions~$\alpha^{(S)}$ and $\beta^{(S)}$, $S \in \{L,R \}$, that minimize the terms
\[
 ||\alpha^{(L)}+1 ||^{2}_{L^{2}([0,1])} +  ||\alpha^{(R)} -1 ||^{2}_{L^{2}([0,1])}
\]
and
\[
  ||\beta^{(L)} ||^{2}_{L^{2}([0,1])} +  ||\beta^{(R)} ||^{2}_{L^{2}([0,1])},
\]
see~\cite{KaSaTa17c}.

\subsection{The $C^1$ isogeometric space $\UV$ and the subspace $\UW$} \label{subsec:spaceW}

We recall the concept of $C^1$ isogeometric spaces over analysis-suitable $G^1$ two-patch geometries studied in~\cite{CoSaTa16,KaSaTa17a}, and especially 
focus on a specific subspace of the entire space of $C^1$ isogeometric functions.

The space~$\UV$ of $C^1$ isogeometric spline functions on $\Omega$ (with respect to the two-patch geometry~$\f{F}$ and spline 
space~$\US{p}{r}$) is given by
\begin{equation} \label{eq:spaceV}
\UV = \{ \phi \in C^1(\Omega) \ :\  \phi \circ \f{F}^{(S)} \in \US{p}{r}\otimes\US{p}{r} , S \in \{L,R \} \}. 
\end{equation}
A function $\phi: \Omega \rightarrow \R$ belongs to the space $\UV$ if and only if the functions $f^{(S)}=\phi \circ \f{F}^{(S)}$, $S \in \{L,R \}$, satisfy that 
\begin{equation} \label{eq:isogeometric}
f^{(S)} \in \US{p}{r} \otimes \US{p}{r}, \quad S \in \{L,R \},
\end{equation}
\begin{equation} \label{eq:trace}
 f^{(L)}(0,\xi_2) = f^{(R)}(0,\xi_2), \quad \xi_2 \in [0,1],
\end{equation}
and
\[
 \alpha^{(R)}(\xi_2) \partial_1 f^{(L)}(0,\xi_2) - \alpha^{(L)}(\xi_2) \partial_1 f^{(R)}(0,\xi_2) + \beta(\xi_2) \partial_2 f^{(L)}(0,\xi_2) = 0, \quad \xi_2 \in [0,1 ],
\]
where the last equation is due to \eqref{eq:beta_cond} further equivalent to 
\begin{equation} \label{eq:derivative}
\frac{\partial_1 f^{(L)}(0,\xi_2) - \beta^{(L)}(\xi_2) \partial_2 f^{(L)}(0,\xi_2)}{\alpha^{(L)}(\xi_2)}  = 
\frac{\partial_1 f^{(R)}(0,\xi_2) - \beta^{(R)}(\xi_2) \partial_2 f^{(R)}(0,\xi_2)}{\alpha^{(R)}(\xi_2)} , \quad \xi_2 \in [0,1],
\end{equation}
see e.g.~\cite{CoSaTa16,Pe15,KaViJu15}. Therefore, the space~$\UV$ can be also described as
\begin{equation}
 \UV = \{ \phi: \Omega \rightarrow \R \ :\ f^{(S)} = \phi \circ \f{F}^{(S)}, S \in \{L,R \}, 
\mbox{ fulfill the equations \eqref{eq:isogeometric}-\eqref{eq:derivative}} \}. 
\end{equation}

Note that the equally valued terms in~\eqref{eq:derivative} represent a specific directional derivative of $\phi$ across the interface~$\Gamma$. 
In fact, recalling that $f^{(S)}= \phi \circ \f{F}^{(S)}$ for $S \in \{L,R\}$, we have
\begin{equation} \label{eq:specific_derivative}
\nabla \phi \cdot (\f{d} \circ \f{F}_0(\xi_2)) = \nabla \phi \cdot (\f{d}^{(S)} \circ \f{F}_0(\xi_2)) =  \frac{\partial_1 f^{(S)}(0,\xi_2) - 
\beta^{(S)}(\xi_2) \partial_2 f^{(S)}(0,\xi_2)}{\alpha^{(S)}(\xi_2)}, \quad \xi_2 \in [0,1], 
\end{equation} 
where $\f{d}$ is a transversal vector to $\Gamma$ given by $\f{d}=\f{d}^{(L)} = \f{d}^{(R)}$ with 
$\f{d}^{(S)} \circ \f{F}_0(\xi_2) = (\partial_1 \f{F}^{(S)}(0,\xi_2),\partial_2 \f{F}^{(S)}(0,\xi_2)) (1,-\beta^{(S)}(\xi_2))^T \frac{1}{\alpha^{(S)}(\xi_2)}$, 
$S \in \{L,R \}$, see~\cite{CoSaTa16,KaSaTa17a}.

The structure and the dimension of the space~$\UV$ heavily depends on the functions~$\alpha^{(L)}$, $\alpha^{(R)}$ and $\beta$, and was fully 
analyzed in~\cite{KaSaTa17a} by computing a basis and its dimension for all possible configurations. 
Below, we restrict ourselves to a simpler subspace~$\UW$ (motivated by~\cite{KaSaTa17c}), which preserves the approximation properties of $\UV$,  
and whose dimension is independent of the functions~$\alpha^{(L)}$, $\alpha^{(R)}$ and $\beta$.

The $C^1$ isogeometric space~$\UW$ is defined as
\[
\UW = \text{span}\, \PhiB , \qquad \PhiB = \PhiB_{\Omega^{(L)}} \cup \PhiB_{\Omega^{(R)}} \cup \PhiB_{\Gamma_0} \cup \PhiB_{\Gamma_1},
\]
with
\begin{equation}\label{eq:Phi-Omega}
\PhiB_{\Omega^{(S)}} = \left\{ \phiS{S}{i}{j} \ : \ i \in \I \setminus \{0,1 \}; \, j \in \I \right\} , \qquad S \in \{L,R \},
\end{equation}
\begin{equation}\label{eq:Phi-Gamma}
 \PhiB_{\Gamma_0} =  \left\{ \phiT{0}{i} \ : \ i \in \I_0 \right\}, \qquad 
 \PhiB_{\Gamma_1} =  \left\{  \phiT{1}{i} \ : \ i \in \I_1  \right\},
\end{equation}
where the functions $\phiS{S'}{i}{j}$, $\phiT{0}{i}$ and $\phiT{1}{i}$ are defined via
\begin{equation}\label{eq:stdbasis}
 \Big(\phiS{S'}{i}{j} \circ \f{F}^{(S)}\Big)(\xi_1,\xi_2) = 
 \begin{cases} \UN{p}{r}{i}(\xi_1)  \UN{p}{r}{j}(\xi_2) & \mbox{if }S=S', \\
  0 & \mbox{otherwise},
\end{cases}
\quad i \in \I \setminus \{0,1 \}; \, j \in \I; \; S, S' \in \{ L , R \},
\end{equation}
\begin{eqnarray}
\Big(\phiT{0}{i} \circ \f{F}^{(S)}\Big)(\xi_1,\xi_2) & = & \UN{p}{r+1}{i}(\xi_2) 
 \Big(\UN{p}{r}{0}(\xi_1) + \UN{p}{r}{1}(\xi_1)  \Big) \notag \\
 & + & \beta^{(S)}(\xi_2) \Big(\UN{p}{r+1}{i} \Big)'(\xi_2) \frac{\tauN{1}}{p} \UN{p}{r}{1}(\xi_1), \quad i \in \I_0; \; S \in \{L,R \}, \label{eq:expr_phi0}
\end{eqnarray}
and
\begin{equation}\label{eq:expr_phi1}
\Big(\phiT{1}{i} \circ \f{F}^{(S)}\Big)(\xi_1,\xi_2)  = \alpha^{(S)}(\xi_2) 
 \UN{p-1}{r}{i}(\xi_2) \UN{p}{r}{1}(\xi_1) , \quad i \in \I_1; \; S \in \{L,R \}.
\end{equation}
The construction of the functions $\phiS{S'}{i}{j}$, $\phiT{0}{i}$ and $\phiT{1}{i}$ guarantees that they are linearly 
independent and therefore form a basis of the space~$\UW$. In addition, the functions fulfill equations 
\eqref{eq:isogeometric}-\eqref{eq:derivative} which implies that they are $C^1$-smooth on $\Omega$, and hence $\UW \subseteq \UV$.
Note that the basis functions $\phiS{S'}{i}{j}$ are standard tensor-product B-splines whose support is included in one of the two patches, while the functions
$\phiT{0}{i}$ and $\phiT{1}{i}$ are combinations of standard B-splines and their support crosses the 
interface $\Gamma$ (see Figure \ref{basisfig1} for an example).

Moreover, the traces and specific directional derivatives~\eqref{eq:specific_derivative} of the functions $\phiT{0}{i}$ and $\phiT{1}{i}$ at the interface~$\Gamma$ are equal to
\[
\phiT{0}{i} \circ \f{F}_{0}(\xi_2) =  \UN{p}{r+1}{i}(\xi_2), \quad  \phiT{1}{i} \circ \f{F}_{0}(\xi_2) =  0,
\]
and
\[
\nabla \phiT{0}{i} \cdot (\f{d} \circ \f{F}_0(\xi_2)) = 0, \quad 
\nabla \phiT{1}{i} \cdot (\f{d} \circ \f{F}_0(\xi_2)) = \UN{p-1}{r}{i}(\xi_2).
\]
Therefore, the $C^1$ isogeometric space~$\UW$ can be also characterized as
\begin{equation} \label{eq:charspaceW}
 \UW  =  \{ \phi \in \UV \ :\ \phi \circ \f{F}_{0}(\xi_2) \in \US{p}{r+1} \mbox{ and } \nabla \phi \cdot (\f{d} \circ \f{F}_{0}(\xi_2)) \in \US{p-1}{r} \}.
\end{equation}

\begin{figure}[h!]\begin{center}
\subfigure[]{\includegraphics[trim=35mm 0mm 30mm 0mm,clip,scale=0.35]{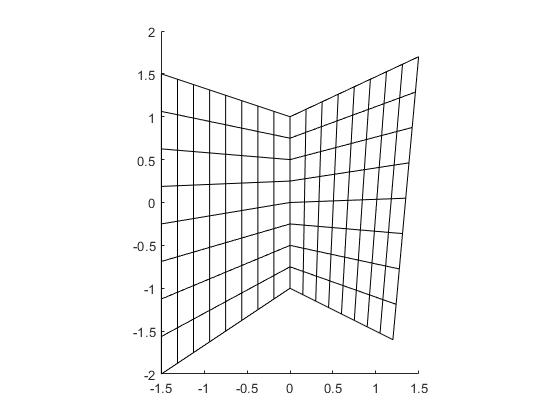}}\hspace{5mm}\\
\subfigure[]{\includegraphics[trim=0mm 0mm 0mm 0mm,clip,scale=0.09]{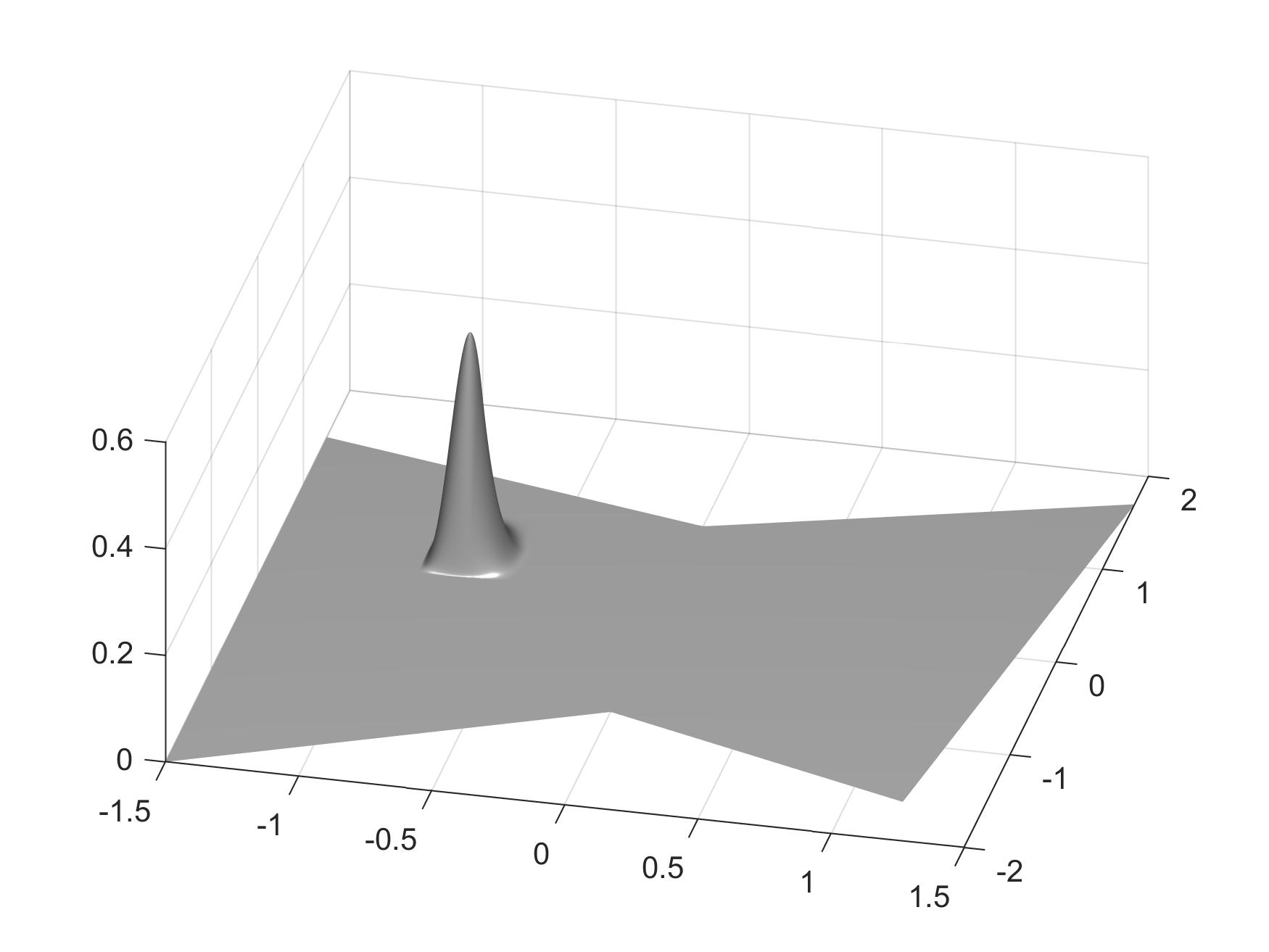}}\hspace{5mm}
\subfigure[]{\includegraphics[trim=0mm 0mm 0mm 0mm,clip,scale=0.09]{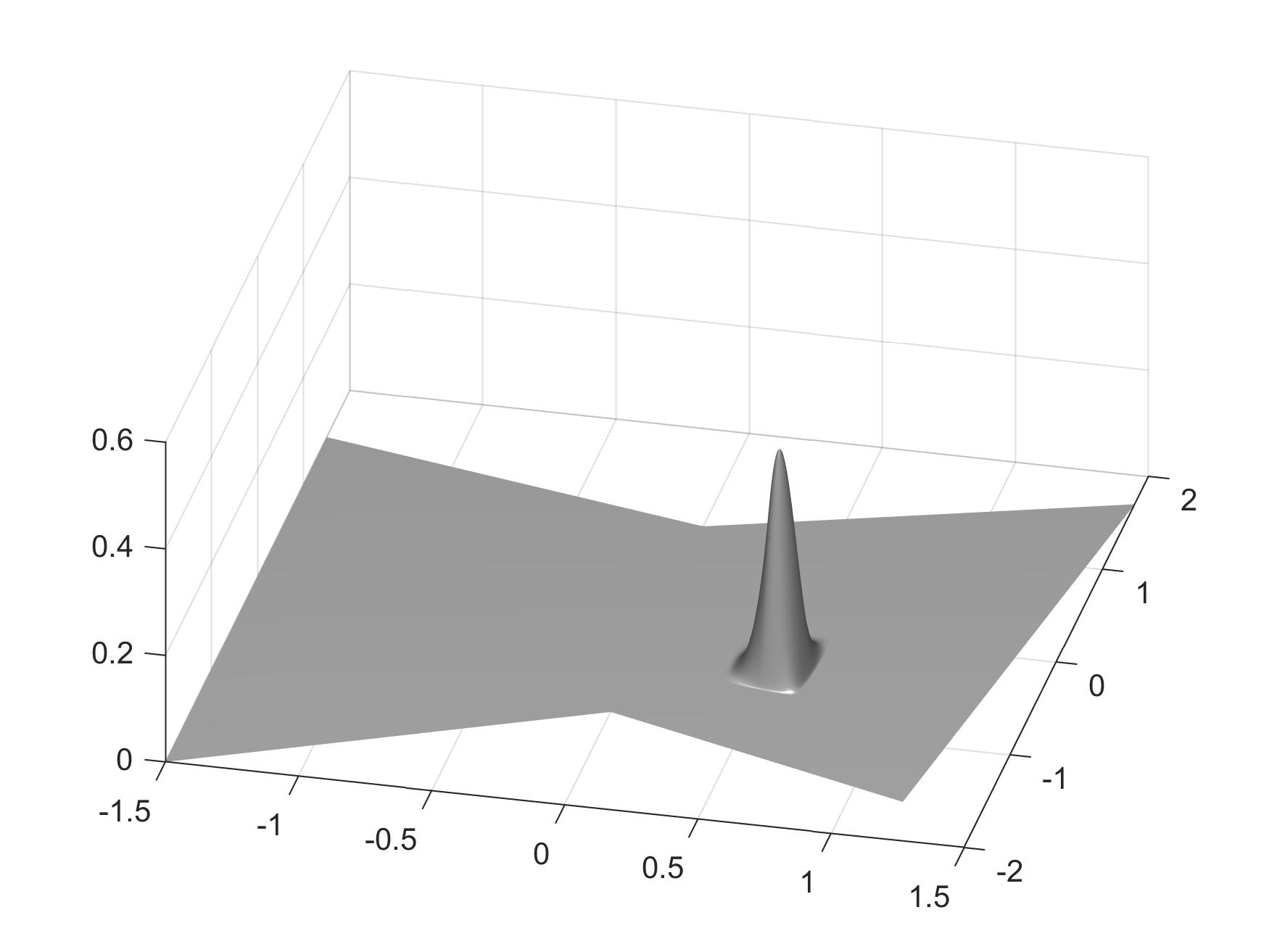}}\\
\subfigure[]{\includegraphics[trim=0mm 0mm 0mm 0mm,clip,scale=0.09]{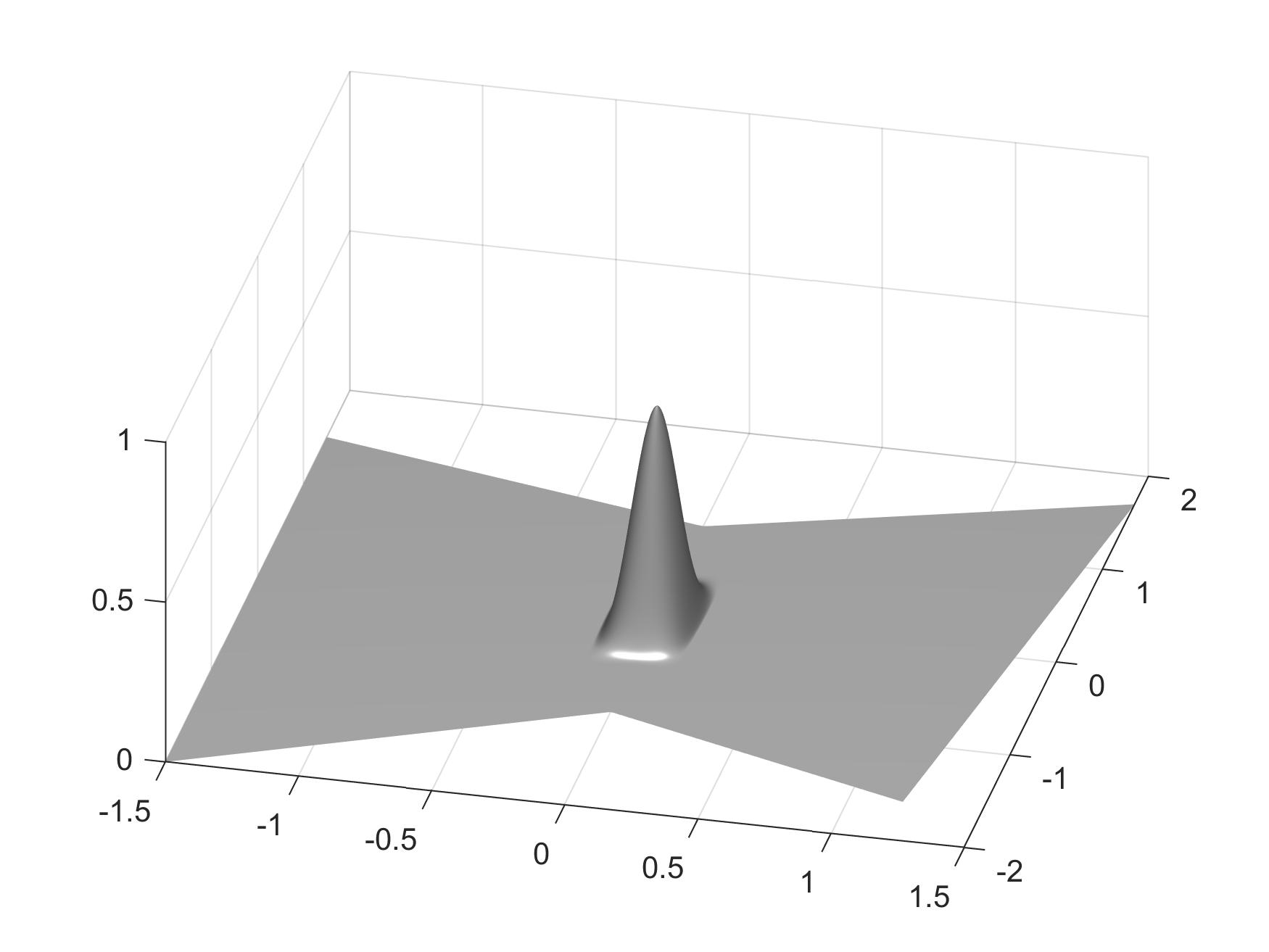}}\hspace{5mm}
\subfigure[]{\includegraphics[trim=0mm 0mm 0mm 0mm,clip,scale=0.09]{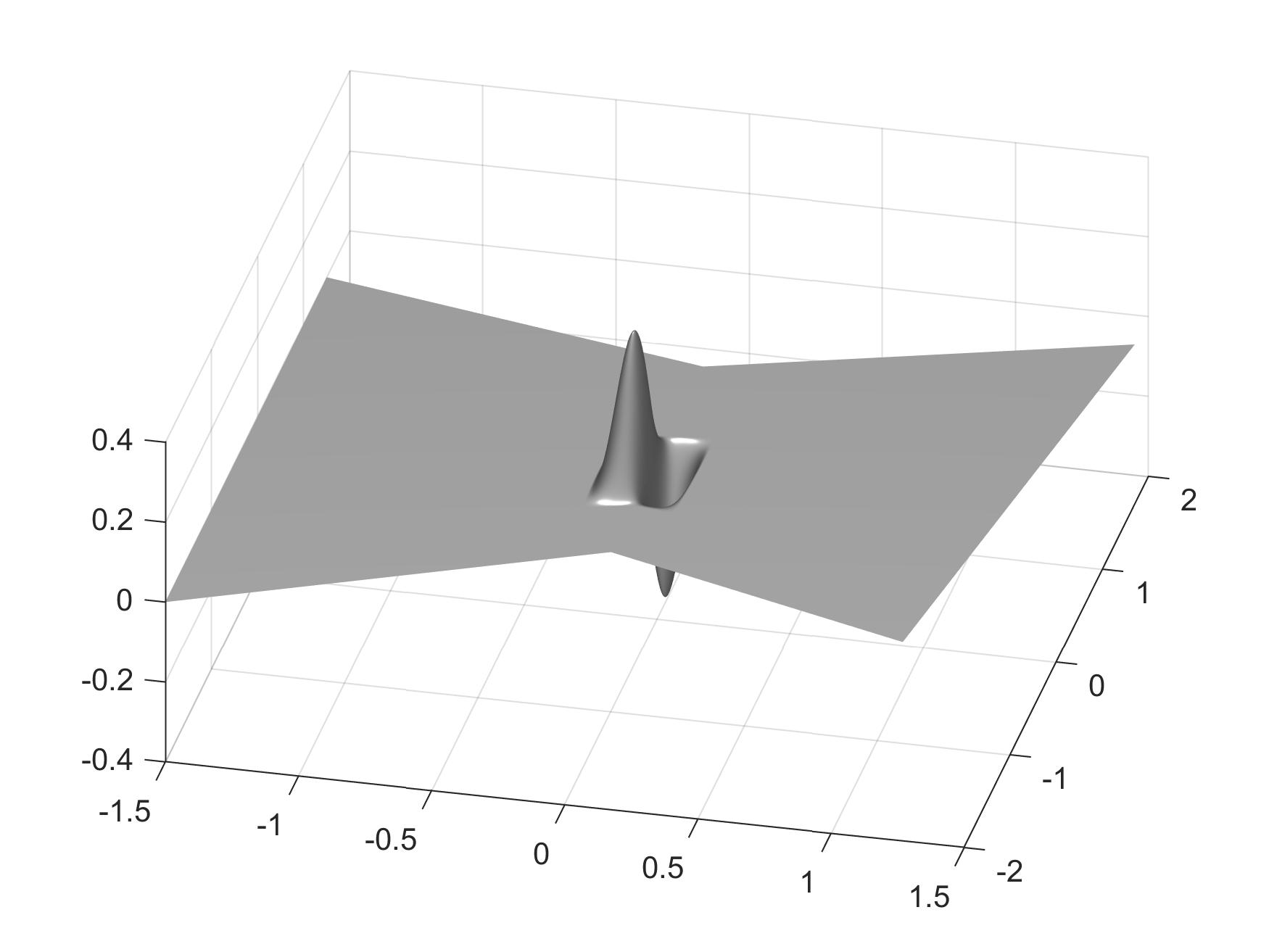}}
\caption{Example of basis functions of $\UW$ on the two-patch domain (a): figures (b)-(c) show two basis functions of type \eqref{eq:stdbasis} (standard B-splines whose support is 
included in one of the two patches), while figures (d) and (e) correspond to basis functions of type \eqref{eq:expr_phi0} and \eqref{eq:expr_phi1}, respectively (whose supports 
intersect the interface).}
 \label{basisfig1}
\end{center}
\end{figure}

\subsection{Representation of the basis with respect to $\US{p}{r} \otimes \US{p}{r}$} \label{subsec:representation}  

We describe the strategy shown in~\cite{KaSaTa17a} to represent the spline functions $\phiS{S'}{i}{j} \circ \f{F}^{(S)}$, $\phiT{0}{i} \circ \f{F}^{(S)}$ and 
$\phiT{1}{i} \circ \f{F}^{(S)}$, $S \in \{L,R \}$, with respect to the spline space $\US{p}{r} \otimes \US{p}{r}$, using a vectorial notation. 
Let us first introduce the vectors of functions $\VTN{0}$, $\VTN{1}$ and $\VTN{2}$, given by
\[
 \VTN{0}(\xi_1,\xi_2)=[\UN{p}{r}{0}(\xi_1)\UN{p}{r}{j}(\xi_2)]_{j \in \I}, \qquad
 \VTN{1}(\xi_1,\xi_2)=[\UN{p}{r}{1}(\xi_1)\UN{p}{r}{j}(\xi_2)]_{j \in \I},
\]
and
\[
\VTN{2}(\xi_1,\xi_2)=[\UN{p}{r}{2}(\xi_1)\UN{p}{r}{0}(\xi_2),\ldots,\UN{p}{r}{2}(\xi_1)\UN{p}{r}{n-1}(\xi_2),\ldots,\UN{p}{r}{n-1}(\xi_1)\UN{p}{r}{n-1}(\xi_2)]^{T},
\]
which represent the whole basis of $\US{p}{r} \otimes \US{p}{r}$. Let us also introduce, the vectors of functions
\[
\vec{\phi}_{\Gamma_0}(\f{x}) = [\phiT{0}{i}(\f{x})]_{i \in \I_0}, \qquad \vec{\phi}_{\Gamma_1}(\f{x}) = [\phiT{1}{i}(\f{x})]_{i \in \I_1},
\]
\[
\vec{\phi}_{\Omega^{(S)}}(\f{x}) = [\phiS{S}{i}{j}(\f{x})]_{i \in \I \setminus \{0,1 \}; \, j \in \I} \quad \text{ for } S \in \{L,R \},
\]
and finally, for $S \in \{L,R\}$, the vectors of functions $\hatPhiT{S}{0}$, $\hatPhiT{S}{1}$, $\hatPhiS{S}{S}$, given by
\[
 \hatPhiT{S}{0}(\xi_1,\xi_2)=[\phiT{0}{i} \circ \f{F}^{(S)}(\xi_1,\xi_2)]_{i \in \I_0}, \qquad
 \hatPhiT{S}{1}(\xi_1,\xi_2)=[\phiT{1}{i} \circ \f{F}^{(S)}(\xi_1,\xi_2)]_{i \in \I_1},
\]
\[
\hatPhiS{S}{S}(\xi_1,\xi_2) = [\phiS{S}{i}{j}\circ \f{F}^{(S)}(\xi_1,\xi_2)]_{i \in \I \setminus \{0,1 \}; \, j \in \I}.
\]
Since the basis functions~$\phiS{S}{i}{j}$ are just the ``standard'' isogeometric functions, the spline functions $\hatPhiS{S}{S}(\xi_1,\xi_2)$ automatically belong to the basis 
of the spline space $\US{p}{r} \otimes \US{p}{r}$, while an analysis of the basis functions in $\hatPhiT{S}{0}(\xi_1,\xi_2)$ and $\hatPhiT{S}{1}(\xi_1,\xi_2)$, leads to the 
following representation
\begin{equation} \label{eq:relation_basic}
\left[\begin{array}{c} \hatPhiT{S}{0}(\xi_1,\xi_2)\\ \hatPhiT{S}{1}(\xi_1,\xi_2)\\ \hatPhiS{S}{S}(\xi_1,\xi_2) \end{array}\right]= 
\left[\begin{array}{ccc} \widehat{B}& \widetilde{B}^{(S)}& 0\\ 0 & \overline{B}\mbox{}^{(S)} & 0 \\ 
0 & 0 & I_{n(n-2)} \end{array}\right]
\left[\begin{array}{c} \VTN{0}(\xi_1,\xi_2)\\ \VTN{1}(\xi_1,\xi_2)\\ \VTN{2}(\xi_1,\xi_2) \end{array}\right],\quad  S \in \{L,R \},
\end{equation}
where $I_m$ denotes the identity matrix of dimension~$m$, and the other blocks of the matrix take the form $\widehat{B}=[\coefa{}{i,j}]_{i \in \I_0,j \in \I}$, 
$\widetilde{B}^{(S)}=[\coefb{(S)}{i,j}]_{i \in \I_0,j \in \I}$, and $\overline{B}\mbox{}^{(S)}= [\coefc{(S)}{i,j}]_{i \in \I_1,j \in \I}$. In fact, these are sparse matrices, and by 
defining the index sets
\[
 \J{0}{i} = \{ j \in \I  \ : \ \text{supp} (\UN{p}{r}{j}) \cap \text{supp} (\UN{p}{r+1}{i}) \neq \emptyset \}, \quad \text{ for } i \in \I_0,
\]
and
\[
 \J{1}{i} = \{ j \in \I  \ : \ \text{supp} (\UN{p}{r}{j}) \cap \text{supp} (\UN{p-1}{r}{i}) \neq \emptyset \}, \quad \text{ for } i \in \I_1,
\]
it can be seen that the possible non-zero entries are limited to $\coefa{}{i,j}$, $\coefb{(S)}{i,j}$, 
$i \in \I_0$, $j \in \J{0}{i}$, and $\coefc{(S)}{i,j}$, $i \in \I_1$, $j \in \J{1}{i}$, respectively.

For the actual computation of these coefficients, let us denote by $\zetaN{m}$, with $m \in \I$, the Greville abscissae of the univariate spline space $\US{p}{r}$. Then, for 
each $S \in \{L ,R \}$ and for each $i \in \I_0$ or $i \in \I_1$, the linear factors 
$\coefa{}{i,j}$, $\coefb{(S)}{i,j}$, $j \in \J{0}{i}$, 
and $\coefc{(S)}{i,j}$, $j \in \J{1}{i}$,
can be obtained by solving the following systems of linear equations
\begin{equation*}
  \Big(\phiT{0}{i} \circ \f{F}^{(L)}\Big)(0,\zetaN{m}) = \sum_{j \in \J{0}{i}} \coefa{}{i,j} \UN{p}{r}{j}(\zetaN{m}) , \quad  m \in \J{0}{i},
 \end{equation*}
 \begin{equation*}
  \frac{\tauN{1} \partial_1 \Big(\phiT{0}{i} \circ \f{F}^{(S)}\Big)(0,\zetaN{m})}{p}  +\Big(\phiT{0}{i} \circ \f{F}^{(S)}\Big)(0,\zetaN{m}) = 
  \sum_{j \in \J{0}{i}} \coefb{(S)}{i,j} \UN{p}{r}{j}(\zetaN{m})  , \quad  m \in \J{0}{i},
 \end{equation*}
 and
 \begin{equation*}
  \frac{\tauN{1} \partial_1 \Big(\phiT{1}{i} \circ \f{F}^{(L)}\Big)(0,\zetaN{m})}{p} = 
  \sum_{j \in \J{1}{i}} \coefc{(S)}{i,j} \UN{p}{r}{j}(\zetaN{m})  , \quad  m \in \J{1}{i},
 \end{equation*}
 respectively, see~\cite{KaSaTa17a} 
 for more details. Note that the coefficients $\coefa{}{i,j}$, $i \in \I_{0}$, are exactly the spline coefficients of the 
 B-spline~$\UN{p}{r+1}{j}$ for the spline representation with respect to the space~$\US{p}{r}$, and can also be computed by simple knot insertion.

\section{$C^1$ hierarchical isogeometric spaces on two-patch geometries} \label{sec:C1hierarchical}

This section introduces an abstract framework for the construction of the hierarchical spline basis, that is defined 
in terms of a multilevel approach applied to an underlying sequence of spline bases that are \emph{locally linearly independent} and characterized 
by \emph{local and compact supports}. The $C^1$ hierarchical isogeometric spaces on two-patch geometries are then defined by applying the 
hierarchical construction to the $C^1$ isogeometric functions described in the previous section. Particular attention is devoted to the proof of 
local linear independence of the basis functions, cf. Section~\ref{subsec:C1hierarchy}, and to the refinement mask that explicitly identifies a two-scale relation between 
hierarchical functions of two consecutive levels, cf. Section~\ref{subsec:refinement}. Note that, even if the hierarchical framework can be applied with different refinement 
strategies between consecutive refinement levels, we here focus on dyadic refinement, the standard choice in most application contexts. In the following the refinement level $\ell$ 
is denoted as a superscript associated to the corresponding symbol.

\subsection{Hierarchical splines: abstract definition} \label{sec:abstract}

Let $\mathbb{U}^0\subset\mathbb{U}^1\subset\ldots\subset\mathbb{U}^{N-1}$ 
be a sequence of $N$ nested multivariate spline spaces defined on a closed domain $D\subset\mathbb{R}^d$, so that any space $\mathbb{U}^\ell$, for $\ell=0,\ldots,N-1$, is 
spanned by a (finite) basis $\Psi^\ell$ satisfying the following properties.
\begin{itemize}
\item[(P1)] \emph{Local linear independence};
\item[(P2)] \emph{Local and compact support}. 
\end{itemize}
The first property guarantees that for any subdomain $S$, the restrictions of the (non-vanishing) functions $\psi\in\Psi^\ell$ to $S$ are linearly independent. The locality of 
the support instead enables to localize the influence of the basis functions with respect to delimited areas of the domain. Note that the nested nature of the spline spaces 
implies the existence of a two-scale relation between adjacent bases: for any level $\ell$, each basis function in $\Psi^\ell$ can be expressed as linear combination of basis 
functions in $\Psi^{\ell+1}$.

By also considering a sequence of closed nested domains 
\begin{equation}\label{eq:domh}
\Omega^0 \supseteq \Omega^1 \supseteq \ldots \supseteq \Omega^{N-1},
\end{equation}
with $\Omega^0\subseteq D$, we can define a hierarchical  spline basis according to the following definition. 

\begin{defn} \label{dfn:hbasis} 
The hierarchical spline basis ${\cal H}$ with respect to the domain hierarchy \eqref{eq:domh}
is defined as
\[{\cal H} = \left\{\psi\in\Psi^\ell : \supp\!^0\psi \subseteq \Omega^\ell \wedge
\supp\!^0\psi\not\subseteq\Omega^{\ell+1}\right\},
\]
where $\supp\!^0\psi = \supp\psi \cap \Omega^0$. 
\end{defn} 

Note that the basis ${\cal H} = {\cal H}^{N-1}$ can be iteratively constructed as follows. 
\begin{enumerate}
\item ${\cal H}^0 = \left\{\psi\in \Psi^0: \supp\!^0\psi \ne \emptyset \right\}$;
\item for $\ell=0,\ldots,N-2$
\[
{\cal H}^{\ell+1} = {\cal H}_A^{\ell+1} \cup {\cal H}_B^{\ell+1},
\]
where 
\[
{\cal H}_A^{\ell+1} = \left\{\psi \in {\cal H}^\ell: \supp\!^0\psi \not\subseteq \Omega^{\ell+1}\right\}\quad\text{and}\quad
{\cal H}_B^{\ell+1} = \left\{\psi \in {\Psi}^{\ell+1}: \supp\!^0\psi \subseteq \Omega^{\ell+1}\right\} .
\]
\end{enumerate}

The main properties of the hierarchical basis can be summarized as follows.

\begin{prop}\label{prop:hbasis} By assuming that properties (P1)-(P2) hold for the bases $\Psi^\ell$, the hierarchical basis satisfies the following properties:
\begin{itemize}
\item[(i)] the functions in ${\cal H}$ are linearly independent, 
\item[(ii)] the intermediate spline spaces are nested,  
namely $\myspan {\cal H}^\ell \subseteq\myspan{\cal H}^{\ell+1}$,
\item[(iii)] 
given an enlargement of the subdomains $(\widehat{\Omega}^\ell)_{\ell=0,\ldots,\widehat{N}-1}$, with $N \le \widehat N$, such that $\Omega^0=\widehat{\Omega}^0$ and $\Omega^{\ell} 
\subseteq \widehat{\Omega}^\ell$, for $\ell=1,\ldots,N-1$, then $\mathrm{span}{\cal H} \subseteq \mathrm{span}\widehat {\cal H}$.
\end{itemize}
\end{prop}
\begin{proof}
The proof follows along the same lines as in \cite{vuong2011} for hierarchical B-splines.
\end{proof}

Proposition~\ref{prop:hbasis} summarizes the key properties of a hierarchical set of basis functions constructed according to Definition~\ref{dfn:hbasis}, when the underlying 
sequence of bases $\Psi^\ell$ satisfies only properties (P1)-(P2).

The results in Proposition~1 remain valid when additional assumptions are considered~\cite{giannelli2014}. In particular, if the basis functions in $\Psi^\ell$, 
for $\ell=0,\ldots,N-1$ are \emph{non-negative}, the hierarchical basis functions are also non-negative. Moreover, the partition of unity property in the hierarchical 
setting can be recovered  by considering the \emph{truncated basis} for hierarchical spline spaces \cite{giannelli2014}. In this case, the partition of unity property at 
each level $\ell$ is also required together with the positiveness of the coefficients in the refinement mask. Even if the construction of $C^1$ functions on two patch 
geometries considered in the previous section does not satisfy the non-negativity and partition of unity properties, we could still apply the truncation mechanism to 
reduce the support of coarser basis functions in the $C^1$ hierarchical basis. Obviously, the resulting truncated basis would not satisfy the other interesting properties 
of truncated hierarchical B-splines, see \cite{giannelli2012,giannelli2014}.

\subsection{The $C^1$ hierarchical isogeometric space} \label{subsec:C1hierarchy}

By following the construction for the $C^1$ isogeometric spline space presented in Section~\ref{sec:C1space}, we can now introduce its hierarchical extension. We recall that 
instead of considering the full $C^1$ space $\UV$ at any hierarchical level, we may restrict to the simpler subspace $\UW$, whose dimension does not depend on the 
functions $\alpha^{(L)}$, $\alpha^{(R)}$ and $\beta$, and it has analogous approximation properties as the full space.

We consider an initial knot vector $\UXI{p}{r,0} \equiv \UXI{p}{r}$ as defined in \eqref{eq:kv} for then introducing the sequence of knot vectors with respect to a fixed degree $p$
\[
\UXI{p}{r,0}, \UXI{p}{r,1} \ldots, \UXI{p}{r,N-1}, 
\]
where each knot vector
\begin{equation*}
\UXI{p}{r,\ell} = \{\underbrace{0,\ldots,0}_{(p+1)-\mbox{\scriptsize times}},
\underbrace{\tauN{1}^{\ell},\ldots ,\tauN{1}^{\ell}}_{(p-r) - \mbox{\scriptsize times}}, 
\underbrace{ \tauN{2}^{\ell},\ldots ,\tauN{2}^{\ell}}_{(p-r) - \mbox{\scriptsize times}},\ldots, 
\underbrace{ \tauN{k^{\ell}}^{\ell},\ldots ,\tauN{k^{\ell}}^{\ell}}_{(p-r) - \mbox{\scriptsize times}},
\underbrace{1,\ldots,1}_{(p+1)-\mbox{\scriptsize times}}\},  
\end{equation*}
for $\ell=1,\ldots,N-1$, is obtained via dyadic refinement of the knot vector of the previous level, keeping the same degree and regularity, and therefore $k^\ell = 2 k^{\ell-1} + 1$. 
We denote by $\US{p}{r,\ell}$ the 
univariate spline space in $[0,1]$ with respect to the open knot vector $ \UXI{p}{r,\ell}$, and let $\UN{p}{r,\ell}{i}$, for $i \in \I^\ell = \{ 0,\ldots,p+k^\ell(p-r) \}$, 
be the associated B-splines. In addition, as in the one-level case, $\US{p}{r+1,\ell}$ and $\US{p-1}{r,\ell}$ ($\UN{p}{r+1,\ell}{i}$ and $\UN{p-1}{r,\ell}{i}$) indicate the subspaces 
(and their basis functions) of higher regularity and lower degree, respectively. We also denote by 
\[
 n^\ell= p+1+ k^\ell(p-r), \;
 \mbox{ }n_0^\ell= p+1+ k^\ell(p-r-1), \;\mbox{ and }\;
 n_1^\ell= p + k^\ell(p-r-1),
\]
the dimensions of the spline spaces $\US{p}{r,\ell}$, $\US{p}{r+1,\ell}$ and $\US{p-1}{r,\ell}$, respectively, and, analogously to $\I^\ell$, we introduce the index sets 
\[
 \I_0^\ell = \{0, \ldots, n_{0}^\ell-1  \}, \qquad
 \I_1^\ell=\{0,\ldots, n_{1}^\ell-1 \},
\]
corresponding to functions in $\US{p}{r+1,\ell}$ and $\US{p-1}{r,\ell}$, respectively.

Let 
\[ 
\UV^0\subset\UV^1\subset\ldots\subset\UV^{N-1}
\]
be a sequence of nested $C^1$ isogeometric spline spaces, with $\UV^\ell$ defined on the two-patch domain 
$\Omega = \Omega^{(L)} \cup \Omega^{(R)}$ with respect to the spline space of level $\ell$. Analogously to the construction detailed in Section~\ref{subsec:spaceW}, for each level 
$0 \le \ell \le N-1$ let us consider the subspace
\[
\UW^\ell = \mathrm{span} \Phi^\ell, \quad \text{ with } \Phi^\ell = \PhiB_{\Omega^{(L)}}^\ell \cup \PhiB_{\Omega^{(R)}}^\ell \cup  
\PhiB_{\Gamma_0}^\ell \cup \PhiB_{\Gamma_1}^\ell,
\]
where the basis functions are given by
\begin{equation*}
\PhiB_{\Omega^{(S)}}^\ell = \left\{ \phiS{S}{i}{j} \ : \ i \in \I^\ell \setminus \{0,1 \}; \, j \in \I^\ell \right\} , 
\quad
 \PhiB_{\Gamma_0}^\ell =  \left\{ \phiT{0}{i} \ : \ i \in \I_0^\ell \right\}, \quad 
 \PhiB_{\Gamma_1}^\ell =  \left\{  \phiT{1}{i} \ : \ i \in \I_1^\ell  \right\},
\end{equation*}
with $S \in \{L,R \}$, directly defined as in \eqref{eq:Phi-Omega} and \eqref{eq:Phi-Gamma} for the one-level case.

By considering a domain hierarchy as in \eqref{eq:domh} on the two-patch domain $\Omega \equiv \Omega^0$, and the sets of isogeometric functions $\Phi^\ell$ at different 
levels, we arrive at the following definition.

\begin{defn}\label{dfn:hC1space} The $C^1$ hierarchical isogeometric space $\UW_H$  with respect to a domain hierarchy of the two-patch domain $\Omega$, that satisfies 
\eqref{eq:domh} with $\Omega^0=\Omega$,
is defined as
\[
\UW_H = \myspan{\cal W}
\quad\text{with}\quad
{\cal W} = \left\{\phi\in\Phi^\ell : \supp\!^0\phi \subseteq \Omega^\ell \wedge
\supp\!^0\phi\not\subseteq\Omega^{\ell+1}\right\}.\]
\end{defn}

In the remaining part of this section we want to prove that ${\cal W}$ is indeed a basis of the $C^1$ hierarchical isogeometric space $\UW_H$. This requires to verify 
the properties for the abstract definition given in Section~\ref{sec:abstract}, in particular the nestedness of the spaces $\UW^\ell$, and that the one-level $C^1$ bases 
spanning each $\UW^{\ell}$, for $\ell=0,\ldots,N-1$, satisfy the hypotheses of Proposition~\ref{prop:hbasis}, i.e. properties (P1)-(P2).
The nestedness of the spaces~$\UW^{\ell}$, $\ell=0,1, \ldots, N-1$, easily follows from definition \eqref{eq:charspaceW}, as stated in the following Proposition.

\begin{prop}\label{prop:nest}
Let $N \in \N$. The sequence of spaces $\UW^{\ell}$, $\ell=0,1,\ldots, N-1$, is nested, i.e.
 \[
  \UW^{0} \subset \UW^{1}  \subset \ldots \subset \UW^{N-1}.
 \]
\end{prop}

\begin{proof}
Let $\ell = 0, \ldots, N-2$, and $\phi \in \UW^\ell \subset \UV^\ell$. By definition \eqref{eq:spaceV} 
the spaces $\UV^{\ell}$ are nested, hence $\phi \in \UV^{\ell} \subset \UV^{\ell+1}$. Since the spline spaces $\US{p}{r+1,\ell}$ and $\US{p-1}{r,\ell}$ are nested, too, we 
have $\phi \circ \f{F}_{0} \in \US{p}{r+1,\ell} \subset \US{p}{r+1,\ell+1}$ and $\nabla \phi \cdot (\f{d} \circ \f{F}_{0}) \in \US{p-1}{r,\ell} \subset \US{p-1}{r,\ell+1}$, 
which implies that $\phi \in \UW^{\ell+1}$.
\end{proof}

The locality and compactness of the support of these functions in (P2) comes directly by construction and by the same property for standard B-splines, 
see \eqref{eq:stdbasis}-\eqref{eq:expr_phi1} and Figure~\ref{basisfig1}. The property of local linear independence in (P1) instead is proven in the following Proposition.
\begin{prop}\label{prop:lli}
The set of basis functions $ \PhiB^\ell = \PhiB_{\Omega^{(L)}}^\ell \cup \PhiB_{\Omega^{(R)}}^\ell \cup \PhiB_{\Gamma_0}^\ell \cup \PhiB_{\Gamma_1}^\ell,$ 
is locally linearly independent, for $\ell=0,\ldots,N-1$.
\end{prop}
\begin{proof} Since we have to prove the statement for any hierarchical level $\ell$, we just remove the superscript $\ell$ in the proof to simplify the notation.
Recall that the functions in $\PhiB$ are linearly independent. It is well known that the functions in $\PhiB_{\Omega^{(L)}} \cup \PhiB_{\Omega^{(R)}}$ are locally linearly independent, 
as they are (mapped) standard B-splines. Furthermore, it is also well known, or easy to verify, that each of the following sets of univariate functions is locally linearly independent
\begin{itemize} 
 \item[(a)] $\{ \UN{p}{r}{0} + \UN{p}{r}{1},\UN{p}{r}{1}\} \cup\{ \UN{p}{r}{i} \}_{i \in \I \setminus \{0,1 \}}$, 
 \item[(b)] $\{\UN{p}{r+1}{i} \}_{i \in \I_0}$,
 \item[(c)] $\{\UN{p-1}{r}{i} \}_{i \in \I_1}$.
\end{itemize}
We prove that the set of functions $\Phi$ is locally linearly independent, which means that, for any open 
set $\widetilde{\Omega} \subset \Omega$ the functions of $\Phi$ that do not vanish in $\widetilde{\Omega}$ are linearly independent 
on $\widetilde{\Omega}$. Let $\widetilde{\I}_{0} \subset \I_{0}$, $\widetilde{\I}_{1} \subset \I_{1}$ and $\widetilde{\I}^{(S)}_{j} \subset \I$, $j \in \I \setminus \{0,1 \}$, 
$S \in \{L,R \}$, be the sets of indices corresponding to those functions~$\phiT{0}{i}$, $\phiT{1}{i}$ and $\phiS{S}{j}{i}$, respectively, that do not vanish 
on~$\widetilde{\Omega}$. Then the equation
\begin{equation} \label{eq:refphi}
 \sum_{i\in \widetilde{\I}_0} \mu_{0,i} \phiT{0}{i}(\f{x}) + \sum_{i\in \widetilde{\I}_1} \mu_{1,i} \phiT{1}{i}(\f{x}) + 
 \sum_{S \in \{L,R \}} \sum_{j \in \I \setminus \{0,1 \} } \sum_{i\in \widetilde{\I}^{(S)}_j} \mu^{(S)}_{j,i} \phiS{S}{j}{i}(\f{x}) =0  , \mbox{ }\f{x} \in \widetilde{\Omega}
\end{equation}
has to imply $\mu_{0,i}=0$ for all $i \in \widetilde{\I}_0$, $\mu_{1,i}=0$ for all $i \in \widetilde{\I}_1$, and $\mu^{(S)}_{j,i}=0$ for all 
$i \in \widetilde{\I}^{(S)}_{j}$, $j \in \I \setminus \{0,1 \}$, $S \in \{L,R \}$. Equation~\eqref{eq:refphi} implies that
\begin{align*} 
  &\sum_{i\in \widetilde{\I}_0} \mu_{0,i} \Big(\phiT{0}{i} \circ \f{F}^{(S)}\Big)(\xi_1,\xi_2) + 
  \sum_{i\in \widetilde{\I}_1} \mu_{1,i} \Big(\phiT{1}{i} \circ \f{F}^{(S)}\Big)(\xi_1,\xi_2) \\
  & + \sum_{j \in \I \setminus \{0,1 \} } \sum_{i\in \widetilde{\I}^{(S)}_j} \mu^{(S)}_{j,i} \Big( \phiS{S}{j}{i} \circ \f{F}^{(S)} \Big) (\xi_1,\xi_2) =0 , 
\end{align*}
for $(\xi_1,\xi_2) \in \widetilde{\Omega}^{(S)}$ and $S \in \{L,R \}$, where $\widetilde{\Omega}^{(S)} \subseteq (0,1)^2$ are the corresponding parameter domains 
for the geometry mappings $\f{F}^{(S)}$ such that the closure of $\widetilde \Omega$ is
\begin{equation*}
\mathrm{cl} (\widetilde{\Omega})=\mathrm{cl}\left( \f{F}^{(L)}(\widetilde{\Omega}^{(L)}) \cup \f{F}^{(R)}(\widetilde{\Omega}^{(R)}) \right). 
\end{equation*}
By substituting the functions $\phiT{0}{i} \circ \f{F}^{(S)}$, $\phiT{1}{i} \circ \f{F}^{(S)}$ and $\phiS{S}{j}{i} \circ \f{F}^{(S)}$ by their corresponding expressions, we obtain
\begin{align*}
  &\sum_{i\in \widetilde{\I}_0} \mu_{0,i}\left(\UN{p}{r+1}{i}(\xi_2) \Big( \UN{p}{r}{0}(\xi_1) + \UN{p}{r}{1}(\xi_1)\Big) +
  \beta^{(S)}(\xi_2) \Big(\UN{p}{r+1}{i}\Big)'(\xi_2) \frac{\tauN{1}}{p}\UN{p}{r}{1}(\xi_1) \right)  \\
  & + \sum_{i\in \widetilde{\I}_1} \mu_{1,i} \left( \alpha^{(S)}(\xi_2)  \UN{p-1}{r}{i}(\xi_2) \UN{p}{r}{1}(\xi_1) \right) +
  \sum_{j \in \I \setminus \{0,1 \} } \sum_{i \in \widetilde{\I}^{(S)}_{j}} \mu_{j,i}^{(S)} \UN{p}{r}{j}(\xi_1)  \UN{p}{r}{i}(\xi_2) =0,
\end{align*}
for $(\xi_1,\xi_2) \in \widetilde{\Omega}^{(S)}$ and $S \in \{L,R \}$, which can be rewritten as
\begin{align}
& \Big(\UN{p}{r}{0}(\xi_1)+\UN{p}{r}{1}(\xi_1)\Big) \Big(\sum_{i\in \widetilde{\I}_0} \mu_{0,i}\UN{p}{r+1}{i}(\xi_2)\Big)  + \UN{p}{r}{1}(\xi_1) 
 \Big( \frac{\tauN{1}}{p} \sum_{i\in \widetilde{\I}_0} \mu_{0,i} \beta^{(S)}(\xi_2) \Big(\UN{p}{r+1}{i}\Big)'(\xi_2) \Big) \label{eq:fullterm} \\
 & + \UN{p}{r}{1}(\xi_1) \Big( \sum_{i\in \widetilde{\I}_1} \mu_{1,i} \alpha^{(S)}(\xi_2)  \UN{p-1}{r}{i}(\xi_2) \Big) + 
 \sum_{j \in \I \setminus \{0,1 \} } \UN{p}{r}{j}(\xi_1) \Big( \sum_{i \in \widetilde{\I}^{(S)}_{j}} \mu_{j,i}^{(S)}   \UN{p}{r}{i}(\xi_2) \Big) =0. \nonumber
\end{align}
Now, since $\widetilde \Omega$ and $\widetilde \Omega^{(S)}$ are open, for each $i \in \widetilde{\I}_0$ there exists a 
point $({\xi}_{1}^{(S)},{\xi}_{2}^{(S)}) \in \widetilde{\Omega}^{(S)}$, with $S \in \{L, R\}$, such that $\phiT{0}{i}$ does not vanish 
in a neighborhood ${Q} \subset \widetilde \Omega^{(S)}$ of the point. Due to the fact that the univariate functions $\UN{p}{r}{0}+\UN{p}{r}{1}$, $\UN{p}{r}{1}$ and 
$\UN{p}{r}{j}$, $j \in \I \setminus \{0,1 \}$ are locally linearly independent and that 
$\UN{p}{r}{0}({\xi}_{1}^{(S)}) +\UN{p}{r}{1}({\xi}_{1}^{(S)}) \neq 0$, we get that 
\begin{equation*} 
 \sum_{i \in \widetilde{\I}_0} \mu_{0,i}\UN{p}{r+1}{i}(\xi_2)=0, \mbox{ }\text{for }\xi_2 \text{ such that } (\xi^{(S)}_{1},\xi_2) \in Q.
\end{equation*}
This equation and the local linear independence of the univariate functions $\{\UN{p}{r+1}{i} \}_{i\in \widetilde{\I}_0}$ imply that $\mu_{0,i} = 0$. Applying 
this argument for all $i \in \widetilde{\I}_0$, we obtain $\mu_{0,i}=0$, $i \in \widetilde{\I}_0$, and the term~\eqref{eq:fullterm} simplifies to
\begin{equation} \label{eq:term_simpl1}
 \UN{p}{r}{1}(\xi_1) \Big( \sum_{i\in \widetilde{\I}_1} \mu_{1,i} \alpha^{(S)}(\xi_2)  \UN{p-1}{r}{i}(\xi_2) \Big) + 
 \sum_{j \in \I \setminus \{0,1 \} } \UN{p}{r}{j}(\xi_1) \Big( \sum_{i \in \widetilde{\I}^{(S)}_{j}} \mu_{j,i}^{(S)}   \UN{p}{r}{i}(\xi_2) \Big) =0. 
\end{equation}
Similarly, we can obtain for each $i \in \widetilde{I}_{1}$  
\begin{equation} \label{eq:termwithalpha}
 \sum_{i\in \widetilde{\I}_{1}} \mu_{1,i} \, \alpha^{(S)}(\xi_2) \UN{p-1}{r}{i}(\xi_2)=0, \mbox{ }\text{for }\xi_2 \text{ such that } (\xi^{(S)}_{1},\xi_2) \in Q,
\end{equation}
with the corresponding points~$(\xi_1^{(S)},\xi_2) \in \widetilde{\Omega}$ and neighborhoods~$Q \subset \widetilde{\Omega}$. 
Since the function $\alpha^{(S)}$ is just a linear function which never takes the value zero, see~\eqref{eq:alpha_cond}, equation~\eqref{eq:termwithalpha} implies that
\[
  \sum_{i\in \widetilde{\I}_{1}} \mu_{1,i} \, \UN{p-1}{r}{i}(\xi_2)=0, \mbox{ }\text{for }\xi_2 \text{ such that } (\xi^{(S)}_{1},\xi_2) \in Q.
\]
The local linear independence of the univariate functions $\{ \UN{p-1}{r}{i} \}_{i\in \widetilde{\I}_{1}}$ implies as before that 
$\mu_{1,i}=0$, $i \in \widetilde{\I}_{1}$, and therefore
the term~\eqref{eq:term_simpl1} simplifies further to 
\[
\sum_{j \in \I \setminus \{0,1 \} } \UN{p}{r}{j}(\xi_1) \Big( \sum_{i \in \widetilde{\I}^{(S)}_{j}} \mu_{j,i}^{(S)}   \UN{p}{r}{i}(\xi_2) \Big) =0. 
\]
Finally, $\mu^{(S)}_{j,i}=0$, $i \in \widetilde{\I}^{(S)}_{j}$, $j \in \I \setminus \{0,1 \}$, $S \in \{L,R \}$, follows directly from the fact that the functions in 
$\PhiB_{\Omega^{(L)}} \cup \PhiB_{\Omega^{(R)}}$ are locally linearly independent.
\end{proof}

Finally, we have all what is necessary to prove the main result.
\begin{thm}\label{cor:hbasis}
${\cal W}$ is a basis for the $C^1$ hierarchical space $\UW_H$.
\end{thm}
\begin{proof}
The result holds because the spaces in Definition~\ref{dfn:hC1space} satisfy the hypotheses in Proposition~\ref{prop:hbasis}. In particular, we have the nestedness of the spaces 
by Proposition~\ref{prop:nest}, and for the basis functions in $\Phi^\ell$ the local linear independence (P1) by Proposition~\ref{prop:lli}, and the local and compact support (P2) 
by their definition in \eqref{eq:stdbasis}-\eqref{eq:expr_phi1}.
\end{proof}

\begin{rem}
In contrast to the here considered $C^1$ basis functions for the case of analysis-suitable $G^1$ two-patch geometries, the analogous $C^1$ basis functions for the multi-patch 
case based on~\cite{KaSaTa17c} are, in general, not locally linearly dependent. Due to the amount of notation needed and to their technicality, we do not report here counterexamples, 
but what happens, even in some basic domain configurations, is that the basis functions defined in the vicinity of a vertex may be locally linearly dependent. As a consequence, the 
construction of a hierarchical $C^1$ space requires a different approach, whose investigation is beyond the scope of the present paper. 
\end{rem}

\section{Refinement mask and implementation} \label{sec:implementation}

In this section we give some details about practical aspects regarding the implementation of isogeometric methods based on the hierarchical space $\UW_H$. 
First, we specify the refinement masks, which allow to write the basis functions of $\Phi^\ell$ as linear combinations of the basis functions of $\Phi^{\ell+1}$. The refinement 
masks are important, as they are needed, for instance, for knot insertion algorithms and some operators in multilevel preconditioning. Then, we focus on the implementation of 
the hierarchical space in the open Octave/Matlab software GeoPDEs \cite{vazquez2016}, whose principles can be applied almost identically to any other isogeometric code. 
The implementation employs the refinement masks for the evaluation of basis functions too.

\subsection{Refinement masks} \label{subsec:refinement}
Let us recall the notations and assumptions from Section~\ref{subsec:C1hierarchy} for the multi-level setting of the spline spaces~$\UW^{\ell}$, $\ell=0,1,\ldots, N-1$,  
where the upper index $\ell$ refers to the specific level of refinement. We will use the same upper index in an analogous manner for further notations, 
which have been mainly introduced in Section~\ref{subsec:representation} for the one-level case, such as for the vectors of functions $\VTN{0}$, $\VTN{1}$, $\VTN{2}$ and 
$\hatPhiT{S}{0}$, $\hatPhiT{S}{1}$, $\hatPhiS{S}{S}$, $S \in \{L,R \}$, and for the transformation matrices $\widehat{B}$, $\widetilde{B}^{(S)}$ and $\overline{B}\mbox{}^{(S)}$, 
$S \in \{L,R \}$.

Let $\R_{+}$ be the set of 
non-negative real numbers. Based on basic properties of B-splines, there exist refinement matrices (refinement masks) $\Lambda_{p}^{r,\ell+1} 
\in \R_{+}^{n^{\ell} \times n^{\ell+1}}$, $\Lambda_{p}^{r+1,\ell+1} \in \R_{+}^{n_0^{\ell} \times n_0^{\ell+1}}$ and $\Lambda_{p-1}^{r,\ell+1} \in \R_{+}^{n_1^{\ell} 
\times n_1^{\ell+1}}$ such that
\[
[\HUN{p}{r}{i}{\ell}(\xi)]_{i \in \I^{\ell}} = \Lambda_{p}^{r,\ell+1} [\HUN{p}{r}{i}{\ell+1}(\xi)]_{i \in \I^{\ell+1}},
\]
\[
 [\HUN{p}{r+1}{i}{\ell}(\xi)]_{i \in \I_0^{\ell}} = \Lambda_{p}^{r+1,\ell+1} [\HUN{p}{r+1}{i}{\ell+1}(\xi)]_{i \in \I_0^{\ell+1}},
\]
and
\[
[\HUN{p-1}{r}{i}{\ell}(\xi)]_{i \in \I_1^{\ell}} = \Lambda_{p-1}^{r,\ell+1} [\HUN{p-1}{r}{i}{\ell+1}(\xi)]_{i \in \I_1^{\ell+1}}.
\]
These refinement matrices are banded matrices with a small bandwidth. Furthermore, using an analogous notation to Section~\ref{subsec:representation} for the vectors of functions, 
the refinement mask between the tensor-product spaces $\US{p}{r,\ell} \otimes \US{p}{r,\ell}$ and $\US{p}{r,\ell+1} \otimes \US{p}{r,\ell+1}$ is obtained by refining in each 
parametric direction as a Kronecker product, and can be written in block-matrix form as
\begin{equation} \label{eq:refining_relation}
\left[\begin{array}{c} \VTN{0}^\ell(\xi_1,\xi_2)\\ \VTN{1}^\ell(\xi_1,\xi_2)\\ \VTN{2}^\ell(\xi_1,\xi_2) \end{array}\right]= 
(\Lambda_{p}^{r,\ell+1} \otimes \Lambda_{p}^{r,\ell+1})
\left[\begin{array}{c} \VTN{0}^{\ell+1}(\xi_1,\xi_2)\\ \VTN{1}^{\ell+1}(\xi_1,\xi_2)\\ \VTN{2}^{\ell+1}(\xi_1.\xi_2) \end{array}\right] = 
\left[\begin{array}{ccc} \Theta^{\ell+1}_{00}& \Theta^{\ell+1}_{01}& \Theta_{02}^{\ell+1}\\ 0 & \Theta^{\ell+1}_{11} & \Theta^{\ell+1}_{12} \\ 
0 & 0 & \Theta^{\ell+1}_{22} \end{array}\right]
\left[\begin{array}{c} \VTN{0}^{\ell+1}(\xi_1,\xi_2)\\ \VTN{1}^{\ell+1}(\xi_1,\xi_2)\\ \VTN{2}^{\ell+1}(\xi_1,\xi_2) \end{array}\right] .
\end{equation}
Note that in case of dyadic refinement (as considered in this work), we have $\Theta^{\ell+1}_{02}=0$. 
\begin{prop} \label{lem:refinement}
 It holds that 
 \begin{equation} \label{eq:refinement_basis}
\left[\begin{array}{c} \vec{\phi}^{\ell}_{\Gamma_0}(\f{x})\\  \vec{\phi}^{\ell}_{\Gamma_1}(\f{x})\\ \vec{\phi}^{\ell}_{\Omega^{(L)}}(\f{x}) \\ 
\vec{\phi}^{\ell}_{\Omega^{(R)}}(\f{x})\end{array}\right]
=\left[\begin{array}{cccc} \Lambda_{p}^{r+1,\ell+1}& 0 &  \widetilde{B}^{(L),\ell} \Theta_{12}^{\ell+1}& 
\widetilde{B}^{(R),\ell} \Theta_{12}^{\ell+1}\\
 0 & \frac{1}{2}\Lambda_{p-1}^{r,\ell+1} & \overline{B}\mbox{}^{(L),\ell} \Theta_{12}^{\ell+1}& \overline{B}\mbox{}^{(R),\ell} \Theta_{12}^{\ell+1}\\
0 & 0 & \Theta_{22}^{\ell+1}& 0\\
0 & 0 & 0 & \Theta_{22}^{\ell+1}
\end{array}\right]\,
\left[\begin{array}{c} \vec{\phi}^{\ell+1}_{\Gamma_0}(\f{x})\\  \vec{\phi}^{\ell+1}_{\Gamma_1}(\f{x})\\ \vec{\phi}^{\ell+1}_{\Omega^{(L)}}(\f{x})\\ 
\vec{\phi}^{\ell+1}_{\Omega^{(R)}}(\f{x})\end{array}\right].
\end{equation}
\end{prop}
\begin{proof}
We first show the refinement relation for the functions~$\vec{\phi}^{\ell}_{\Gamma_0}$. For this, let us consider the corresponding spline functions 
$\HhatPhiT{S}{0}{\ell}$, $S \in \{L,R \}$. 
On the one hand, using first relation~\eqref{eq:relation_basic} and then relation~\eqref{eq:refining_relation} with the fact that $\Theta^{\ell+1}_{02}=0$, we obtain
\begin{eqnarray*}
\HhatPhiT{S}{0}{\ell}(\xi_1,\xi_2) & =  \left[\begin{array}{ccc} \widehat{B}^{\ell}& \widetilde{B}^{(S),\ell}& 0
\end{array}\right]
\left[\begin{array}{ccc} \HVTN{0}{\ell}(\xi_1,\xi_2) & \HVTN{1}{\ell}(\xi_1,\xi_2) & \HVTN{2}{\ell}(\xi_1,\xi_2) \end{array}\right]^T \\
& =  \left[\begin{array}{ccc} \widehat{B}^{\ell}& \widetilde{B}^{(S),\ell}& 0
\end{array}\right]
\left[\begin{array}{ccc} \Theta^{\ell+1}_{00}& \Theta^{\ell+1}_{01}& 0\\ 0 & \Theta^{\ell+1}_{11} & \Theta^{\ell+1}_{12} \\ 
0 & 0 & \Theta^{\ell+1}_{22} \end{array}\right]
\left[\begin{array}{c} \VTN{0}^{\ell+1}(\xi_1,\xi_2) \\ \VTN{1}^{\ell+1}(\xi_1,\xi_2)\\ \VTN{2}^{\ell+1}(\xi_1,\xi_2) \end{array}\right] ,
\end{eqnarray*}
which is equal to
\begin{equation} \label{eq:rep1}
\left[ \begin{array}{cc} \widehat{B}^{\ell}\Theta^{\ell+1}_{00} & \widehat{B}^{\ell} \Theta^{\ell+1}_{01} + \widetilde{B}^{(S),\ell} \Theta^{\ell+1}_{11}  
\end{array}\right] 
\left[\begin{array}{c} \VTN{0}^{\ell+1}(\xi_1,\xi_2)\\ \VTN{1}^{\ell+1}(\xi_1,\xi_2)  \end{array}\right] 
+ \widetilde{B}^{(S),\ell} \Theta_{12}^{\ell+1}  \VTN{2}^{\ell+1}(\xi_1,\xi_2).
\end{equation}
On the other hand, the functions~$\HhatPhiT{S}{0}{\ell}$ possess the form
\begin{equation*}
 \HhatPhiT{S}{0}{\ell}(\xi_1,\xi_2)  = \left[\HUN{p}{r+1}{i}{\ell}(\xi_2)\right]_{i \in \I_0^{\ell}}\Big(\HUN{p}{r}{0}{\ell}(\xi_1) +\HUN{p}{r}{1}{\ell}(\xi_1) \Big) + \\  
  \frac{\tau^{\ell}_1}{p} \beta^{(S)}(\xi_2) \left[\Big(\HUN{p}{r+1}{i}{\ell}\Big)'(\xi_2)\right]_{i \in \I_0^{\ell}} \HUN{p}{r}{1}{\ell}(\xi_1).
\end{equation*}
By refining the B-spline functions~$\HUN{p}{r+1}{i}{\ell+1}(\xi_2)$, we obtain
\begin{eqnarray*}
 \HhatPhiT{S}{0}{\ell}(\xi_1,\xi_2)  = &\Lambda_{p}^{r+1,\ell+1}\left[\HUN{p}{r+1}{i}{\ell+1}(\xi_2)\right]_{i \in \I_0^{\ell+1}}\Big(\HUN{p}{r}{0}{\ell}(\xi_1) 
 +\HUN{p}{r}{1}{\ell}(\xi_1) \Big) \\  
 & + \displaystyle \frac{\tau_1^{\ell}}{p} \beta^{(S)}(\xi_2) \Lambda_{p}^{r+1,\ell+1} \left[\Big(\HUN{p}{r+1}{i}{\ell+1}\Big)'(\xi_2)\right]_{i \in \I_0^{\ell+1}} 
 \HUN{p}{r}{1}{\ell}(\xi_1).
\end{eqnarray*}
Then, refining the B-spline functions $\HUN{p}{r}{0}{\ell}(\xi_1)+\HUN{p}{r}{1}{\ell}(\xi_1)$ and $\HUN{p}{r}{1}{\ell}(\xi_1)$ leads to 
\begin{eqnarray*}
&\HhatPhiT{S}{0}{\ell}(\xi_1,\xi_2)  = \Lambda_{p}^{r+1,\ell+1}\left[\HUN{p}{r+1}{i}{\ell+1}(\xi_2)\right]_{i \in \I_0^{\ell+1}}
 \Big(\sum_{j \in \I^{\ell+1}} \lambda^{\ell+1}_{0,j} \HUN{p}{r}{j}{\ell+1}(\xi_1) + \sum_{j \in \I^{\ell+1}} \lambda^{\ell+1}_{1,j} \HUN{p}{r}{j}{\ell+1}(\xi_1) \Big) \nonumber\\ 
 & + \displaystyle \frac{\tau_1^{\ell}}{p} \beta^{(S)}(\xi_2) \Lambda_{p}^{r+1,\ell+1} \left[\Big(\HUN{p}{r+1}{i}{\ell+1}\Big)'(\xi_2)\right]_{i \in \I_0^{\ell+1}} 
 \sum_{j \in \I^{\ell+1}} \lambda_{1,j}^{\ell+1}\HUN{p}{r}{j}{\ell+1}(\xi_1), \nonumber
\end{eqnarray*}
where $\lambda^{\ell+1}_{i,j}$ are the entries of the refinement matrix~$\Lambda_{p}^{r,\ell+1}$. Since we refine dyadically, 
we have $\lambda^{\ell+1}_{0,0}=1$, $\lambda^{\ell+1}_{0,1}=\frac{1}{2}$, $\lambda^{\ell+1}_{1,0}=0$, $\lambda^{\ell+1}_{1,1}=\frac{1}{2}$ and 
$\tau_{1}^{\ell+1}=\frac{\tau_1^{\ell}}{2}$, and we get
\begin{eqnarray*}
 & \HhatPhiT{S}{0}{\ell}(\xi_1,\xi_2)  = \bigg( \Lambda_{p}^{r+1,\ell+1}\left[\HUN{p}{r+1}{i}{\ell+1}(\xi_2)\right]_{i \in \I_0^{\ell+1}}
 \Big(\HUN{p}{r}{0}{\ell+1}(\xi_1) + \HUN{p}{r}{1}{\ell+1}(\xi_1) \Big) \\
 & + \displaystyle \frac{\tau_1^{\ell+1}}{p}\beta^{(S)}(\xi_2) \Lambda_{p}^{r+1,\ell+1} \left[\Big(\HUN{p}{r+1}{i}{\ell+1}\Big)'(\xi_2)\right]_{i \in \I_0^{\ell+1}} 
 \HUN{p}{r}{1}{\ell+1}(\xi_1) \bigg) \\
 & + \bigg( \Lambda_{p}^{r+1,\ell+1}\left[\HUN{p}{r+1}{i}{\ell+1}(\xi_2)\right]_{i \in \I_0^{\ell+1}}
 \Big(\sum_{j \in \I^{\ell+1} \setminus \{0,1 \}} (\lambda^{\ell+1}_{0,j} + \lambda^{\ell+1}_{1,j}) \HUN{p}{r}{j}{\ell+1}(\xi_1) \Big) \\
 & + \displaystyle \frac{\tau_1^{\ell}}{p} \beta^{(S)}(\xi_2) \Lambda_{p}^{r+1,\ell+1} \left[\Big(\HUN{p}{r+1}{i}{\ell+1}\Big)'(\xi_2)\right]_{i \in \I_0^{\ell+1}} 
 \sum_{j \in \I^{\ell+1} \setminus \{0,1 \} } \lambda_{1,j}^{\ell+1}\HUN{p}{r}{j}{\ell+1}(\xi_1) \bigg),
\end{eqnarray*}
which is equal to
\begin{eqnarray}
 & \HhatPhiT{S}{0}{\ell}(\xi_1,\xi_2)  = \Lambda_{p}^{r+1,\ell+1} \HhatPhiT{S}{0}{\ell+1}(\xi_1,\xi_2) \nonumber \\
 & + \bigg( \Lambda_{p}^{r+1,\ell+1}\left[\HUN{p}{r+1}{i}{\ell+1}(\xi_2)\right]_{i \in \I_0^{\ell+1}}
 \Big(\sum_{j \in \I^{\ell+1} \setminus \{0,1 \}} (\lambda^{\ell+1}_{0,j} + \lambda^{\ell+1}_{1,j}) \HUN{p}{r}{j}{\ell+1}(\xi_1) \Big) \label{eq:rep2} \\
 & + \displaystyle \frac{\tau_1^{\ell}}{p} \beta^{(S)}(\xi_2) \Lambda_{p}^{r+1,\ell+1} \left[\Big(\HUN{p}{r+1}{i}{\ell+1}\Big)'(\xi_2)\right]_{i \in \I_0^{\ell+1}} 
 \sum_{j \in \I^{\ell+1} \setminus \{0,1 \} } \lambda_{1,j}^{\ell+1}\HUN{p}{r}{j}{\ell+1}(\xi_1) \bigg). \nonumber
\end{eqnarray}
By analyzing the two equal value terms~\eqref{eq:rep1} and \eqref{eq:rep2} with respect to the spline representation in $\xi_1$-direction formed by the 
B-splines~$\HUN{p}{r}{j}{\ell+1}(\xi_1)$, $j \in \I$, 
one can observe that both first terms and both second terms each must coincide. This leads to
\begin{equation*}
\HhatPhiT{S}{0}{\ell}(\xi_1,\xi_2)  =  \Lambda_{p}^{r+1,\ell+1} \HhatPhiT{S}{0}{\ell+1}(\xi_1,\xi_2) +  \widetilde{B}^{(S),\ell} 
\Theta_{12}^{\ell+1}  \VTN{2}^{\ell+1}(\xi_1,\xi_2),
\end{equation*}
which directly implies the refinement relation for the functions~$\vec{\phi}^{\ell}_{\Gamma_0}$.

The refinement for the functions~$\vec{\phi}^{\ell}_{\Gamma_1}$ can be proven similarly. Considering the spline functions $\HhatPhiT{S}{1}{\ell}$, $S \in \{L,R \}$, we get, 
on the one hand, by using relations~\eqref{eq:relation_basic} and~\eqref{eq:refining_relation} and the fact that $\Theta^{\ell+1}_{02}=0$
\begin{eqnarray}
\HhatPhiT{S}{1}{\ell}(\xi_1,\xi_2) & =  \left[\begin{array}{ccc} 0 & \overline{B}\mbox{}^{(S),\ell} & 0
\end{array}\right]
\left[\begin{array}{ccc} \HVTN{0}{\ell}(\xi_1,\xi_2) & \HVTN{1}{\ell}(\xi_1,\xi_2) & \HVTN{2}{\ell}(\xi_1,\xi_2) \end{array}\right]^T \nonumber \\
& = \left[\begin{array}{ccc} 0 & \overline{B}\mbox{}^{(S),\ell}& 0
\end{array}\right]
\left[\begin{array}{ccc} \Theta^{\ell+1}_{00}& \Theta^{\ell+1}_{01}& 0\\ 0 & \Theta^{\ell+1}_{11} & \Theta^{\ell+1}_{12} \\ 
0 & 0 & \Theta^{\ell+1}_{22} \end{array}\right]
\left[\begin{array}{c} \VTN{0}^{\ell+1}(\xi_1,\xi_2) \\ \VTN{1}^{\ell+1}(\xi_1,\xi_2)\\ \VTN{2}^{\ell+1}(\xi_1,\xi_2) \end{array}\right] \nonumber \\
& = \overline{B}\mbox{}^{(S),\ell} \Theta_{11}^{\ell +1} \VTN{1}^{\ell+1}(\xi_1,\xi_2) + \overline{B}\mbox{}^{(S),\ell} \Theta_{12}^{\ell +1} \VTN{2}^{\ell+1}(\xi_1,\xi_2). 
\label{eq:phi1term1}
\end{eqnarray}
On the other hand, the functions~$\HhatPhiT{S}{1}{\ell}$ can be expressed as
\begin{equation*}
 \HhatPhiT{S}{1}{\ell}(\xi_1,\xi_2)  = \alpha^{(S)}(\xi_2) \, \left[\HUN{p-1}{r}{i}{\ell}(\xi_2)\right]_{i \in \I_1^{\ell}} \HUN{p}{r}{1}{\ell}(\xi_1),
\end{equation*}
and after refining the B-spline functions~$\HUN{p}{r}{1}{\ell}(\xi_1)$ and $\HUN{p-1}{r}{i}{\ell}(\xi_2)$, $i \in \I_1^{\ell}$ we obtain that this is equal to
\[
 \HhatPhiT{S}{1}{\ell}(\xi_1,\xi_2) =  \alpha^{(S)}(\xi_2) \, \Lambda_{p-1}^{r,\ell+1} \left[\HUN{p-1}{r}{i}{\ell+1}(\xi_2)\right]_{i \in \I_1^{\ell +1}} 
 \sum_{j \in \I^{\ell+1}} \lambda_{1,j}^{\ell+1}\HUN{p}{r}{j}{\ell+1}(\xi_1),
\]
where $\lambda^{\ell+1}_{i,j}$ are again the entries of the refinement matrix~$\Lambda_{p}^{r,\ell+1}$. Recalling that $\lambda^{\ell+1}_{1,0}=0$ and $\lambda^{\ell+1}_{1,1}=
\frac{1}{2}$, we get 
\begin{align}
 & \HhatPhiT{S}{1}{\ell}(\xi_1,\xi_2) =  \alpha^{(S)}(\xi_2) \, \Lambda_{p-1}^{r,\ell+1} \left[\HUN{p-1}{r}{i}{\ell+1}(\xi_2)\right]_{i \in \I_1^{\ell +1}} 
\Big(\frac{1}{2} \HUN{p}{r}{1}{\ell+1}(\xi_1) +  \sum_{j \in \I^{\ell+1} \setminus \{0,1 \}} \lambda_{1,j}^{\ell+1}\HUN{p}{r}{j}{\ell+1}(\xi_1) \Big) \nonumber \\
& = \frac{1}{2} \Lambda_{p-1}^{r,\ell+1} \HhatPhiT{S}{1}{\ell +1}(\xi_1,\xi_2) + \alpha^{(S)}(\xi_2) \, \Lambda_{p-1}^{r,\ell+1} 
\left[\HUN{p-1}{r}{i}{\ell+1}(\xi_2)\right]_{i \in \I_1^{\ell +1}} 
 \sum_{j \in \I^{\ell+1} \setminus \{0,1 \}} \lambda_{1,j}^{\ell+1}\HUN{p}{r}{j}{\ell+1}(\xi_1). \label{eq:phi1term2}
\end{align}
Considering the two equal value terms~\eqref{eq:phi1term1} and \eqref{eq:phi1term2}, one can argue as for the case of the functions~$\HhatPhiT{S}{0}{\ell}$, that both first terms 
and both second terms each must coincide. This implies
\[
 \HhatPhiT{S}{1}{\ell}(\xi_1,\xi_2) = \frac{1}{2} \Lambda_{p-1}^{r,\ell+1} \HhatPhiT{S}{1}{\ell +1}(\xi_1,\xi_2) +
 \overline{B}\mbox{}^{(S),\ell} \Theta_{12}^{\ell +1} \VTN{2}^{\ell+1}(\xi_1,\xi_2),
\]
which finally shows the refinement relation for the functions~$\vec{\phi}^{\ell}_{\Gamma_1}$.

Finally, the relation for the functions $\vec{\phi}^{\ell}_{\Omega^{(S)}}$, $S \in \{L,R \}$, directly follows from relation~\eqref{eq:refining_relation}, 
since they correspond to ``standard'' B-splines.
\end{proof}

\subsection{Details about the implementation}

The implementation of GeoPDEs is based on two main structures: the mesh, that contains the information related to the computational geometry and the quadrature, and that did not need 
any change; and the space, with the necessary information to evaluate the basis functions and their derivatives. The new implementation was done in two steps: we first introduced 
the space of $C^1$ basis functions of one single level, as in Section~\ref{subsec:spaceW}, and then we added the hierarchical construction.

For the space of one level, we created a new space structure that contains the numbering for the basis functions of the three different types, namely $\PhiB_{\Omega^{(S)}}, 
\PhiB_{\Gamma_0}$ and $\PhiB_{\Gamma_1}$. The evaluation of the basis functions, and also matrix assembly, is performed using the representation of $C^1$ basis functions in terms 
of standard tensor-product B-splines, as in Section~\ref{subsec:representation}. Indeed, one can first assemble the matrix for tensor-product B-splines, and then multiply on each 
side this matrix by the same matrix given in \eqref{eq:relation_basic}, in the form
\[
K^{(S)}_{\UW} = B^{(S)} K^{(S)}_{\US{}{}} (B^{(S)})^{\top}, \quad \text{ with } B^{(S)} = \left[\begin{array}{ccc} \widehat{B}& \widetilde{B}^{(S)}& 0\\ 
0 & \overline{B}\mbox{}^{(S)} & 0 \\ 
0 & 0 & I_{n(n-2)} \end{array}\right], \text{ for } S = L,R,
\]
where $K^{(S)}_{\US{}{}}$ represents the stiffness matrix for the standard tensor-product B-spline space on the patch $\Omega^{(S)}$, and $K^{(S)}_{\UW}$ is the contribution to 
the stiffness matrix for the $\UW$ space from the same patch. Obviously, the same can be done at the element level, by restricting the matrices to suitable submatrices using the 
indices of non-vanishing functions on the element.

To implement the hierarchical $C^1$ splines we construct the same structures and algorithms detailed in \cite{garau2018}. First, it is necessary to complete the space structure of 
one single level, that we have just described, with some functionality to compute the support of a given basis function, as explained in \cite[Section~5.1]{garau2018}. Second, the 
hierarchical structures are constructed following the description in the same paper, except that for the evaluation of basis functions, and in particular for matrix assembly, 
we make use of the refinement masks of Section \ref{subsec:refinement}. The refinement masks essentially give us the two-level relation required by the algorithms in 
\cite{garau2018}, and in particular the matrix $C_\ell^{\ell+1}$ of that paper, that is used both during matrix assembly and to compute the refinement matrix after enlargement of 
the subdomains.

\section{Numerical examples} \label{sec:examples}

We present now some numerical examples to show the good performance of the hierarchical $C^1$ spaces for their use in combination with adaptive methods. We consider two different 
kinds of numerical examples: the first three tests are run for Poisson problems with an automatic adaptive scheme, while in the last numerical test we solve the bilaplacian problem, 
with a pre-defined refinement scheme.

\subsection{Poisson problem}
 The first three examples are tests on the Poisson equation
\[
\left \{
\begin{array}{rl}
- \Delta u = f & \text{ in } \Omega, \\
u = g & \text{ on } \partial \Omega. 
\end{array}
\right.
\]
The goal is to show that using the $C^1$ space basis does not spoil the properties of the local refinement. The employed isogeometric algorithm is based on the adaptive 
loop (see, e.g., \cite{BrBuGiVa19})
\begin{equation*}
\hbox{SOLVE} \longrightarrow \hbox{ESTIMATE} \longrightarrow \hbox{MARK} \longrightarrow \hbox{REFINE}.
\end{equation*}
In particular, for the examples we {\sl solve} the variational formulation of the problem imposing the Dirichlet boundary condition by Nitsche's method, and the problem is 
to find $u \in \UW_H$ such that
\[
\int_\Omega \nabla u \cdot \nabla v - \int_{\Gamma_D} \frac{du}{dn} v - \int_{\Gamma_D} u \frac{dv}{dn} + \int_{\Gamma_D} \frac{\gamma}{h} u v = 
\int_\Omega f v - \int_{\Gamma_D} g \frac{dv}{dn} + \int_{\Gamma_D} \frac{\gamma}{h} g v \quad \forall v \in \UW_H,
\]
where $h$ is the local element size, and the penalization parameter is chosen as $\gamma = 10(p+1)$, with $p$ the degree. The {\sl error estimate} is computed with a 
residual-based estimator, and the {\sl marking} of the elements at each iteration is done using D\"orfler's strategy 
(when not stated otherwise, we set the marking parameter equal to $0.75$). The {\sl refinement} step of the loop dyadically refines all the marked elements. Although 
optimal convergence can be only proved if we refine using a refinement strategy that guarantees that meshes are admissible \cite{buffa2016c}, previous numerical results 
show also a good behavior of non-admissible meshes \cite{BrBuGiVa19}.

For each of the three examples we report the results for degrees ${\bf p}=(3,3),(4,4)$, with 
$C^1$ smoothness across the interface, and  
with a regularity~$r$ equal to degree minus two within the single patches.
We compare the results for the adaptive scheme with those obtained by refining uniformly, and also with the ones obtained by employing the same adaptive scheme for 
hierarchical spaces with $C^0$ continuity across the interface, while the same regularity within the patches as above is kept.

\begin{ex} \label{ex:Example1}
For the first numerical example we consider the classical L-shaped domain $[-1,1]^2\setminus(0,1)\times(-1,0)$ defined by two patches as depicted in 
Figure~\ref{fig1b}, and the right-hand side $f$ and the boundary condition $g$ are chosen such that the exact solution is given by
\begin{equation*}
u(\rho, \theta)=\rho^{\frac{4}{3}} \, \sin\Bigg(\frac{4}{3}\theta\Bigg),
\end{equation*}
with $\rho$ and $\theta$ the polar coordinates. As it is well known, the exact solution has a singularity at the reentrant corner. 

\begin{figure}[htb!]
\begin{center}
\subfigure[Domain used in the Examples 1 and 4.]{\includegraphics[trim=20mm 0mm 20mm 0mm,clip,scale=0.20]{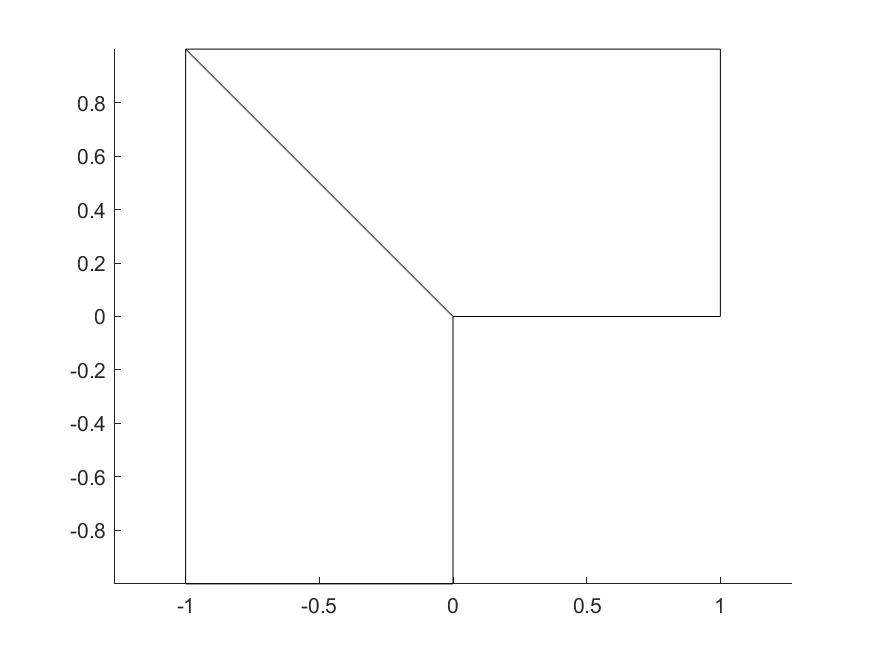} \label{fig1b}}
\subfigure[Domain used in the Examples 2 and 3.]{\includegraphics[trim=20mm 0mm 20mm 0mm,clip,scale=0.20]{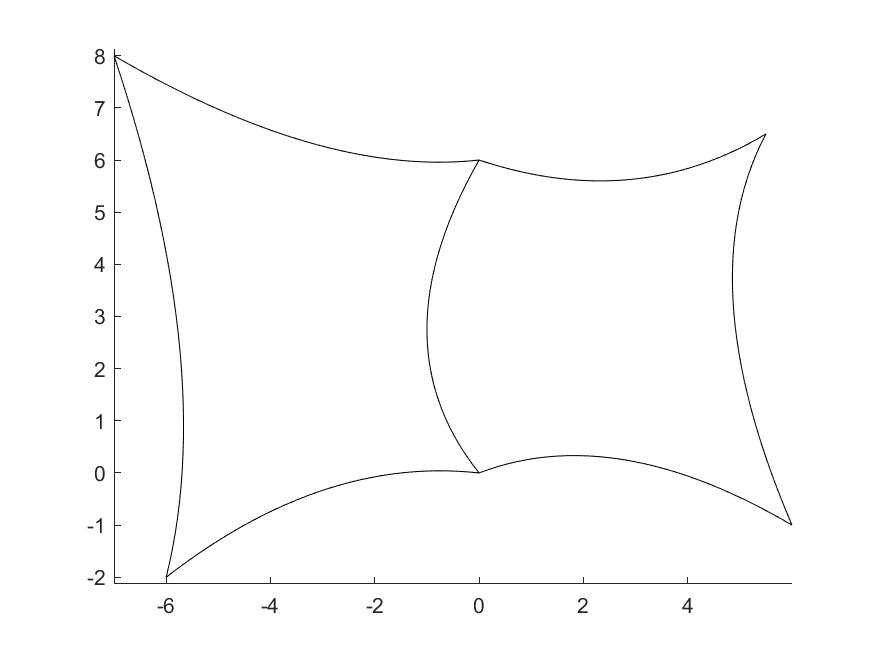} \label{fig1a}} 
\hspace{5mm} 
\end{center}
\caption{The two domains used in the numerical examples.} \label{fig1}
\end{figure}

We start the adaptive simulation with a coarse mesh of $4 \times 4$ elements on each patch, and we use D\"orfler's parameter equal to $0.90$ for the marking of the elements. 
The convergence results are presented in Figure~\ref{fig6}. It can be seen that the error in $H^1$ semi-norm and the estimator converge with the expected rate, in terms of 
the degrees of freedom, both for the $C^1$ and the $C^0$ discretization, and that this convergence rate is better than the one obtained with uniform refinement. Moreover, 
the error for the $C^1$ discretization is slightly lower than the one for the $C^0$ discretization, although they are very similar. This is in good agreement with what has 
been traditionally observed for isogeometric methods: the accuracy per degree of freedom is better for higher continuity. In this case, since the continuity only changes 
near the interface, the difference is very small.

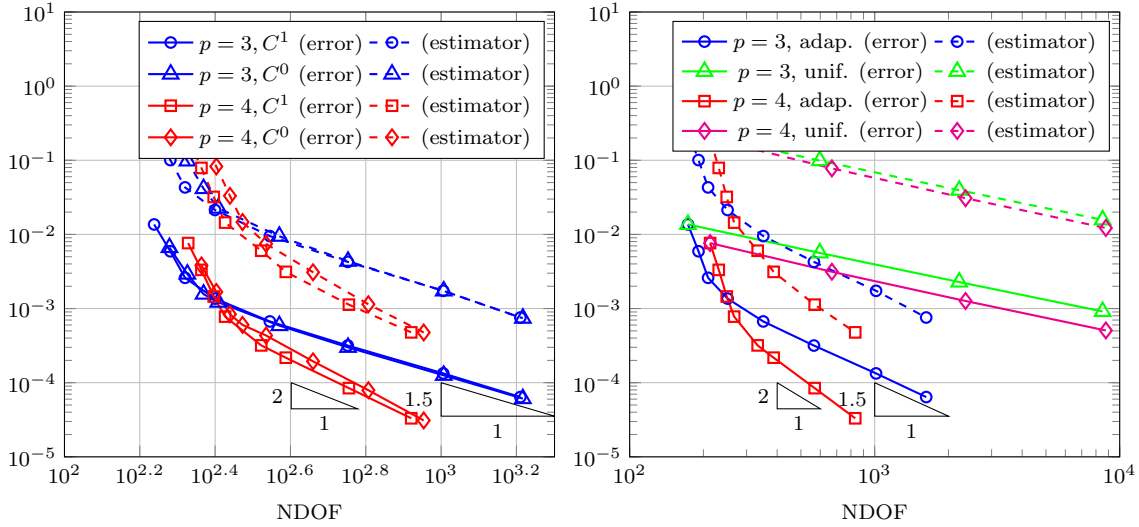
\begin{figure}[ht]
\centering
  \begin{tikzpicture}
  \begin{loglogaxis}[width=.5\textwidth,height=7.5cm,xlabel={NDOF},ylabel={},xmin=1e2,xmax=2e3,ymin=1e-5,ymax=1e1, grid = major, anchor = north, legend columns = 2]
  \addplot[color=blue, mark=o, solid, thick] table [x={ndof}, y={error}, col sep=comma] {ex1_p3_C1.csv};
  \addplot[color=blue, mark=o, dashed, thick, mark options = {solid}] table [x={ndof}, y={estimator}, col sep=comma] {ex1_p3_C1.csv};
  \addplot[color=blue, mark=triangle, solid, thick, mark options = {scale=1.7}] table [x={ndof}, y={error}, col sep=comma] {ex1_p3_C0.csv};
  \addplot[color=blue, mark=triangle, dashed, thick, mark options = {solid, scale=1.7}] table [x={ndof}, y={estimator}, col sep=comma] {ex1_p3_C0.csv};
  \addplot[color=red, mark=square, solid, thick, mark options = {scale=1.}] table [x={ndof}, y={error}, col sep=comma] {ex1_p4_C1.csv};
  \addplot[color=red, mark=square, dashed, thick, mark options = {solid, scale=1.}] table [x={ndof}, y={estimator}, col sep=comma] {ex1_p4_C1.csv};
  \addplot[color=red, mark=diamond, solid, thick, mark options = {scale=1.5}] table [x={ndof}, y={error}, col sep=comma] {ex1_p4_C0.csv};
  \addplot[color=red, mark=diamond, dashed, thick, mark options = {solid, scale=1.5}] table [x={ndof}, y={estimator}, col sep=comma] {ex1_p4_C0.csv};
  \scriptsize \legend{{$p=3, C^1$ (error)}, {(estimator)}, {$p=3, C^0$ (error)}, {(estimator)}, {$p=4, C^1$ (error)}, {(estimator)}, {$p=4, C^0$ (error)}, {(estimator)}}
\convergenceslope{4e2}{6e2}{1e-4}{2}{font=\scriptsize,above}{solid}
\convergenceslope{1e3}{2e3}{1e-4}{1.5}{font=\scriptsize,above}{solid}
  \end{loglogaxis}
\end{tikzpicture}
  \begin{tikzpicture}
  \begin{loglogaxis}[width=.5\textwidth,height=7.5cm,xlabel={NDOF},ylabel={},xmin=1e2,xmax=1e4,ymin=1e-5,ymax=1e1, grid = major, anchor = north, legend columns = 2]
  \addplot[color=blue, mark=o, solid, thick] table [x={ndof}, y={error}, col sep=comma] {ex1_p3_C1.csv};
  \addplot[color=blue, mark=o, dashed, thick, mark options = {solid}] table [x={ndof}, y={estimator}, col sep=comma] {ex1_p3_C1.csv};
  \addplot[color=green, mark=triangle, solid, thick, mark options = {scale=1.7}] table [x={ndof}, y={error}, col sep=comma] {ex1_p3_C1_glob.csv};
  \addplot[color=green, mark=triangle, dashed, thick, mark options = {solid, scale=1.7}] table [x={ndof}, y={estimator}, col sep=comma] {ex1_p3_C1_glob.csv};
  \addplot[color=red, mark=square, solid, thick, mark options = {scale=1.}] table [x={ndof}, y={error}, col sep=comma] {ex1_p4_C1.csv};
  \addplot[color=red, mark=square, dashed, thick, mark options = {solid, scale=1.}] table [x={ndof}, y={estimator}, col sep=comma] {ex1_p4_C1.csv};
  \addplot[color=magenta, mark=diamond, solid, thick, mark options = {scale=1.5}] table [x={ndof}, y={error}, col sep=comma] {ex1_p4_C1_glob.csv};
  \addplot[color=magenta, mark=diamond, dashed, thick, mark options = {solid, scale=1.5}] table [x={ndof}, y={estimator}, col sep=comma] {ex1_p4_C1_glob.csv};
  \scriptsize \legend{{$p=3$, adap. (error)}, {(estimator)}, {$p=3$, unif. (error)}, {(estimator)}, {$p=4$, adap. (error)}, {(estimator)}, {$p=4$, unif. (error)}, {(estimator)}}
\convergenceslope{4e2}{6e2}{1e-4}{2}{font=\scriptsize,above}{solid}
\convergenceslope{1e3}{2e3}{1e-4}{1.5}{font=\scriptsize,above}{solid}
  \end{loglogaxis}
\end{tikzpicture}
\caption{Error in $H^1$ semi-norm and estimator for Example~\ref{ex:Example1} with ${\bf p}=(3,3)$ and ${\bf p}=(4,4)$, compared with $C^0$ case (left) and with global refinement 
case (right).}
 \label{fig6}
\end{figure}

We also show in Figure~\ref{fig5} the final meshes obtained with the different discretizations. It is clear that the adaptive method correctly refines the mesh in the 
vicinity of the reentrant corner, where the singularity occurs, and the refinement gets more local with higher degree.

\begin{figure}[h!]\begin{center}
\subfigure[${\bf p}=(3,3)$, $C^0$ functions on the interface: NDOF=$1648$.]{\includegraphics[trim=40mm 0mm 40mm 0mm,clip,scale=0.25]{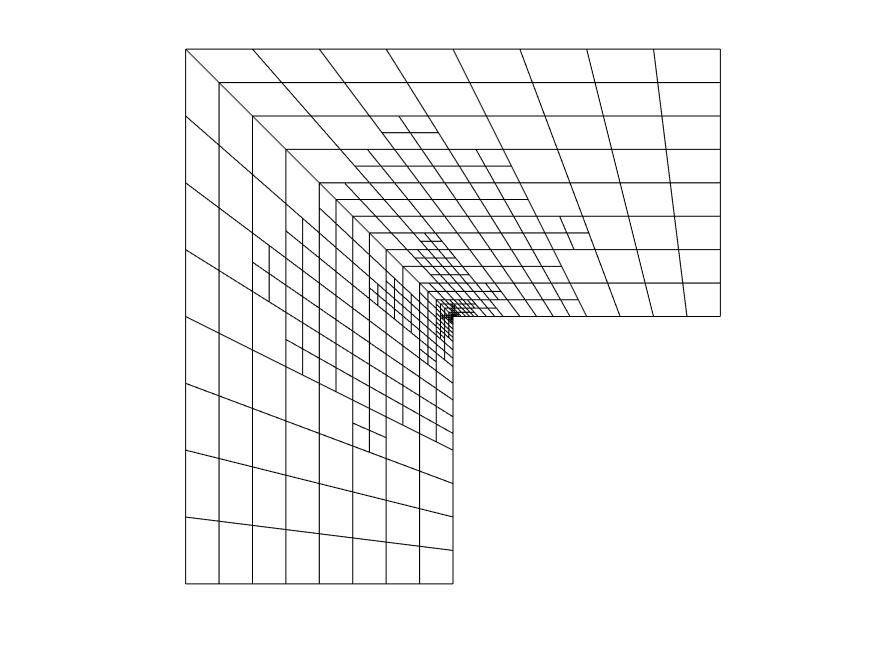}}\hspace{5mm}
\subfigure[${\bf p}=(3,3)$, $C^1$ functions on the interface: NDOF=$1623$.]{\includegraphics[trim=40mm 0mm 40mm 0mm,clip,scale=0.25]{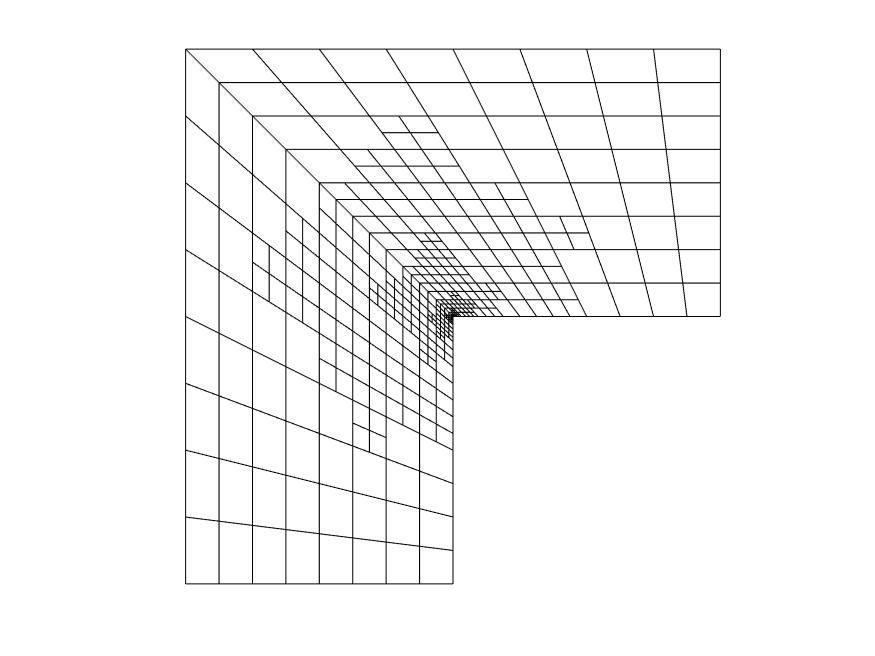}}\\
\subfigure[${\bf p}=(4,4)$, $C^0$ functions on the interface: NDOF=$833$.]{\includegraphics[trim=40mm 0mm 40mm 0mm,clip,scale=0.25]{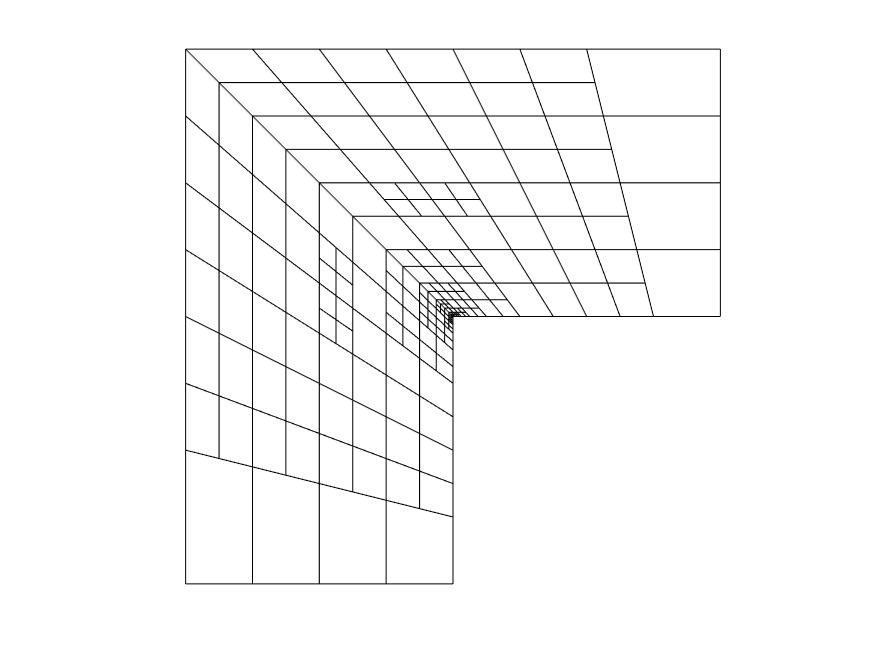}}\hspace{5mm}
\subfigure[${\bf p}=(4,4)$, $C^1$ functions on the interface: NDOF=$833$.]{\includegraphics[trim=40mm 0mm 40mm 0mm,clip,scale=0.25]{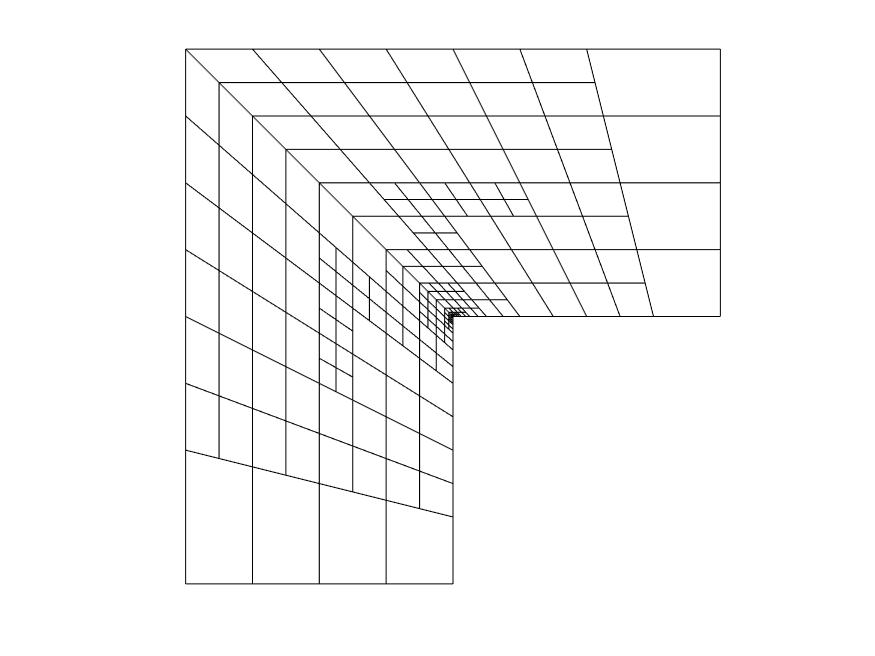}}
\caption{Hierarchical meshes for Example~\ref{ex:Example1}, with ${\bf p}=(3,3)$ and ${\bf p}=(4,4)$. Apparently the meshes are the same for the $C^0$ and $C^1$ case, but there are 
some differences in the finest levels.}
 \label{fig5}
\end{center}
\end{figure}
\end{ex}

\begin{ex} \label{ex:Example2}
In the second example the data of the problem are chosen in such a way that the exact solution is
\begin{equation*}
u(x,y)=(-120x + x^2 - 96y - 8xy + 16y^2)^{12/5} \cos (\pi y/20),
\end{equation*}
defined on the domain shown in Figure~\ref{fig1a}. The geometry of the domain is given by two bicubic B\'ezier patches, and the control points are chosen following the algorithm 
in \cite{KaSaTa17b}, in such a way that the geometry is given by an analysis-suitable $G^1$ parametrization, see~\ref{app:geometry} for details. Note that we have chosen 
the solution such that it has a singularity along the interface.
In this example we start the adaptive simulation with a coarse mesh of $8 \times 8$ elements on each patch. We present the convergence results in Figure~\ref{fig:conv_ex2}. 
As before, both the (relative) error and the estimator converge with optimal rate, and both for the $C^0$ and the $C^1$ discretizations, with slightly better result for the $C^1$ 
spaces. We note that, since the singularity occurs along a line, optimal order of convergence for higher degrees cannot be obtained without anisotropic refinement, as it was observed 
in the numerical examples in \cite[Section~4.6]{gantner-phd}.

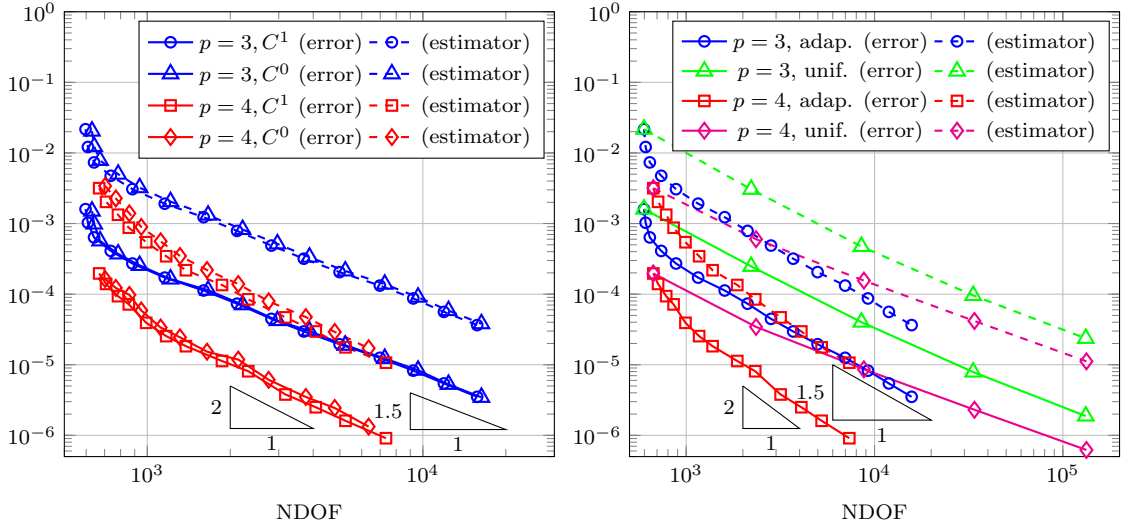
\begin{figure}[ht]
\centering
  \begin{tikzpicture}
  \begin{loglogaxis}[width=.5\textwidth,height=7.5cm,xlabel={NDOF},ylabel={},xmin=5e2,xmax=3e4,ymin=5e-7,ymax=1e0, grid = major, anchor = north, legend columns = 2]
  \addplot[color=blue, mark=o, solid, thick] table [x={ndof}, y={error}, col sep=comma] {ex2_p3_C1.csv};
  \addplot[color=blue, mark=o, dashed, thick, mark options = {solid}] table [x={ndof}, y={estimator}, col sep=comma] {ex2_p3_C1.csv};
  \addplot[color=blue, mark=triangle, solid, thick, mark options = {scale=1.7}] table [x={ndof}, y={error}, col sep=comma] {ex2_p3_C0.csv};
  \addplot[color=blue, mark=triangle, dashed, thick, mark options = {solid, scale=1.7}] table [x={ndof}, y={estimator}, col sep=comma] {ex2_p3_C0.csv};
  \addplot[color=red, mark=square, solid, thick, mark options = {scale=1.}] table [x={ndof}, y={error}, col sep=comma] {ex2_p4_C1.csv};
  \addplot[color=red, mark=square, dashed, thick, mark options = {solid, scale=1.}] table [x={ndof}, y={estimator}, col sep=comma] {ex2_p4_C1.csv};
  \addplot[color=red, mark=diamond, solid, thick, mark options = {scale=1.5}] table [x={ndof}, y={error}, col sep=comma] {ex2_p4_C0.csv};
  \addplot[color=red, mark=diamond, dashed, thick, mark options = {solid, scale=1.5}] table [x={ndof}, y={estimator}, col sep=comma] {ex2_p4_C0.csv};
 \scriptsize \legend{{$p=3, C^1$ (error)}, {(estimator)}, {$p=3, C^0$ (error)}, {(estimator)}, {$p=4, C^1$ (error)}, {(estimator)}, {$p=4, C^0$ (error)}, {(estimator)}}
\convergenceslope{2e3}{4e3}{5e-6}{2}{font=\scriptsize,above}{solid}
\convergenceslope{9e3}{2e4}{4e-6}{1.5}{font=\scriptsize,above}{solid}
  \end{loglogaxis}
\end{tikzpicture}
  \begin{tikzpicture}
  \begin{loglogaxis}[width=.5\textwidth,height=7.5cm,xlabel={NDOF},ylabel={},xmin=5e2,xmax=2e5,ymin=5e-7,ymax=1e0, grid = major, anchor = north, legend columns = 2]
  \addplot[color=blue, mark=o, solid, thick] table [x={ndof}, y={error}, col sep=comma] {ex2_p3_C1.csv};
  \addplot[color=blue, mark=o, dashed, thick, mark options = {solid}] table [x={ndof}, y={estimator}, col sep=comma] {ex2_p3_C1.csv};
  \addplot[color=green, mark=triangle, solid, thick, mark options = {scale=1.7}] table [x={ndof}, y={error}, col sep=comma] {ex2_p3_C1_glob.csv};
  \addplot[color=green, mark=triangle, dashed, thick, mark options = {solid, scale=1.7}] table [x={ndof}, y={estimator}, col sep=comma] {ex2_p3_C1_glob.csv};
  \addplot[color=red, mark=square, solid, thick, mark options = {scale=1.}] table [x={ndof}, y={error}, col sep=comma] {ex2_p4_C1.csv};
  \addplot[color=red, mark=square, dashed, thick, mark options = {solid, scale=1.}] table [x={ndof}, y={estimator}, col sep=comma] {ex2_p4_C1.csv};
  \addplot[color=magenta, mark=diamond, solid, thick, mark options = {scale=1.5}] table [x={ndof}, y={error}, col sep=comma] {ex2_p4_C1_glob.csv};
  \addplot[color=magenta, mark=diamond, dashed, thick, mark options = {solid, scale=1.5}] table [x={ndof}, y={estimator}, col sep=comma] {ex2_p4_C1_glob.csv};
  \scriptsize \legend{{$p=3$, adap. (error)}, {(estimator)}, {$p=3$, unif. (error)}, {(estimator)}, {$p=4$, adap. (error)}, {(estimator)}, {$p=4$, unif. (error)}, {(estimator)}}
\convergenceslope{2e3}{4e3}{5e-6}{2}{font=\scriptsize,above}{solid}
\convergenceslope{6e3}{2e4}{1e-5}{1.5}{font=\scriptsize,above}{solid}
  \end{loglogaxis}
\end{tikzpicture}
\caption{Relative error in $H^1$ semi-norm and corresponding estimator for Example~\ref{ex:Example2} with ${\bf p}=(3,3)$ and ${\bf p}=(4,4)$, compared with $C^0$ case (left) and 
with global refinement case (right).}
 \label{fig:conv_ex2}
\end{figure}

We also present in Figure~\ref{fig:meshes_ex2} the finest meshes obtained with the different discretizations, and it can be observed that the adaptive method correctly refines 
near the interface, where the singularity occurs.

\begin{figure}[htb]\begin{center}
\subfigure[${\bf p}=(3,3)$, $C^0$ functions on the interface: NDOF=$16310$]{\includegraphics[trim=35mm 0mm 30mm 0mm,clip,scale=0.22]{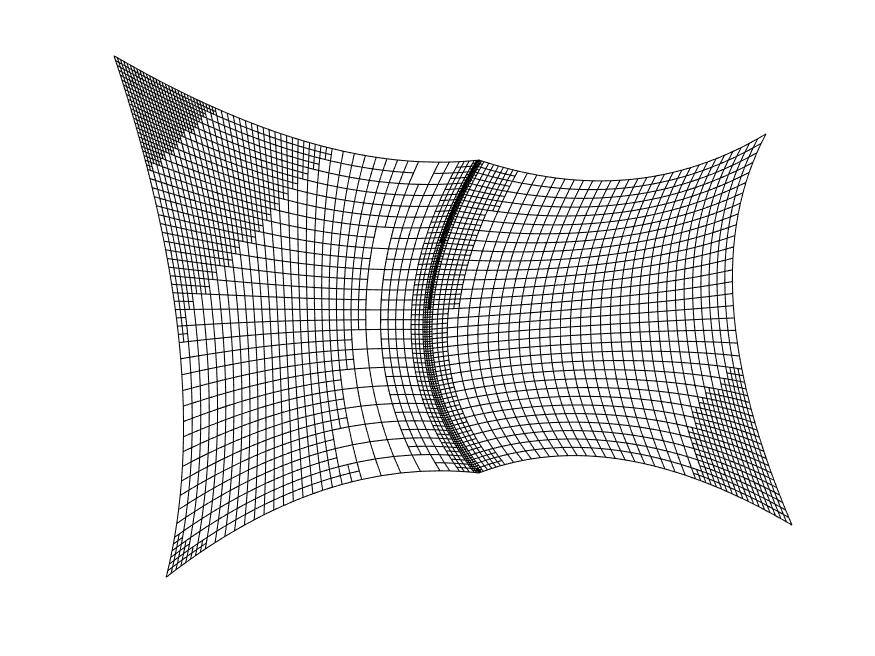}}\hspace{5mm}
\subfigure[${\bf p}=(3,3)$, $C^1$ functions on the interface: NDOF=$15741$]{\includegraphics[trim=35mm 0mm 30mm 0mm,clip,scale=0.22]{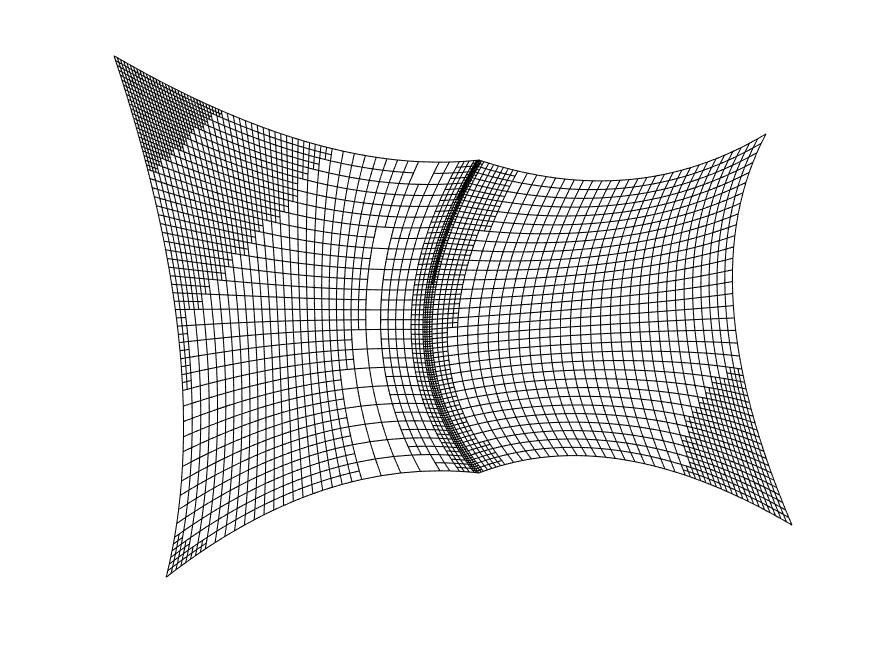}}\\
\subfigure[${\bf p}=(4,4)$, $C^0$ functions on the interface: NDOF=$6357$]{\includegraphics[trim=35mm 0mm 30mm 0mm,clip,scale=0.22]{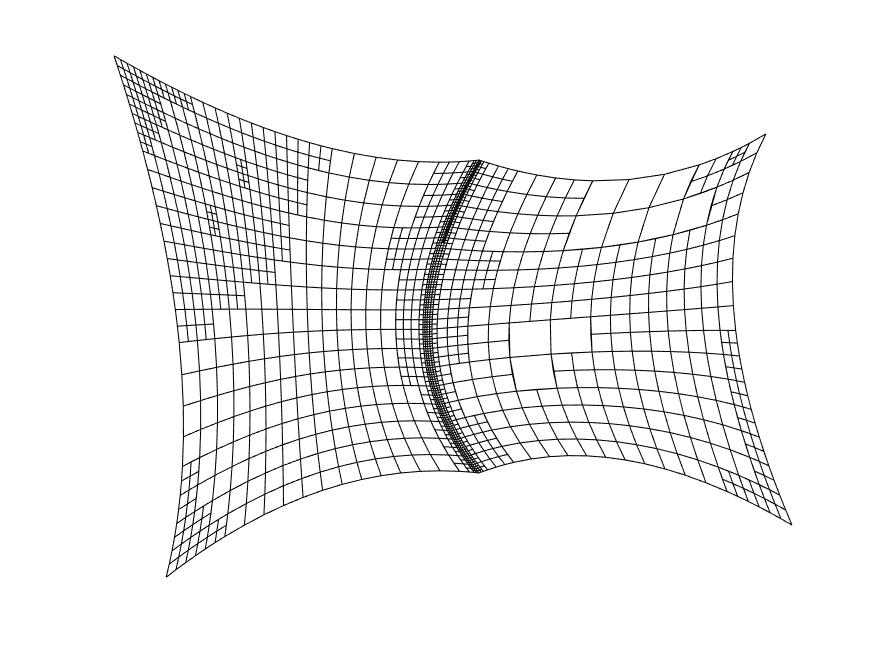}}\hspace{5mm}
\subfigure[${\bf p}=(4,4)$, $C^1$ functions on the interface: NDOF=$7347$]{\includegraphics[trim=35mm 0mm 30mm 0mm,clip,scale=0.22]{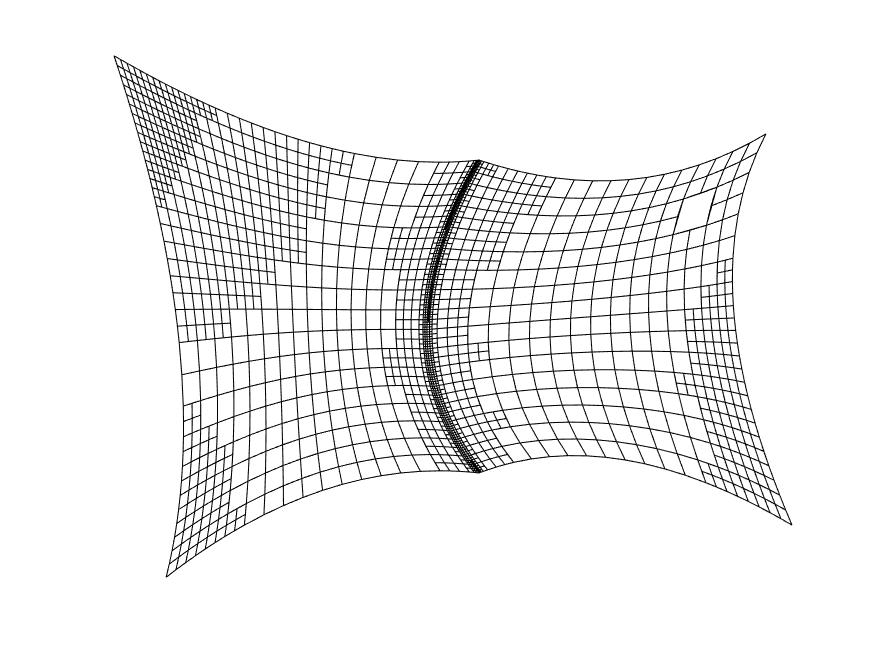}}
\caption{Hierarchical meshes for Example~\ref{ex:Example2}, with ${\bf p}=(3,3)$ and ${\bf p}=(4,4)$.}
 \label{fig:meshes_ex2}
\end{center}
\end{figure}
\end{ex}

\begin{ex} \label{ex:Example3}
We consider the same domain as in the previous example, and the right-hand side and the boundary condition are chosen in such a way that the exact 
solution is given by
\begin{equation*}
u(x,y)=(y-1.7)^{12/5} \cos(x/4).
\end{equation*}
In this case the solution has a singularity along the line $y=1.7$, that crosses the interface and is not aligned with the mesh.

The convergence results, that are presented in Figure~\ref{fig:conv_ex3}, are very similar to the ones of the previous example, and show optimal convergence rates for 
both the $C^1$ and the $C^0$ discretizations. As before, we also present in Figure~\ref{fig:mesh_ex3} the finest meshes obtained with the different discretizations. It is 
evident that the adaptive algorithm successfully refines along the singularity line.

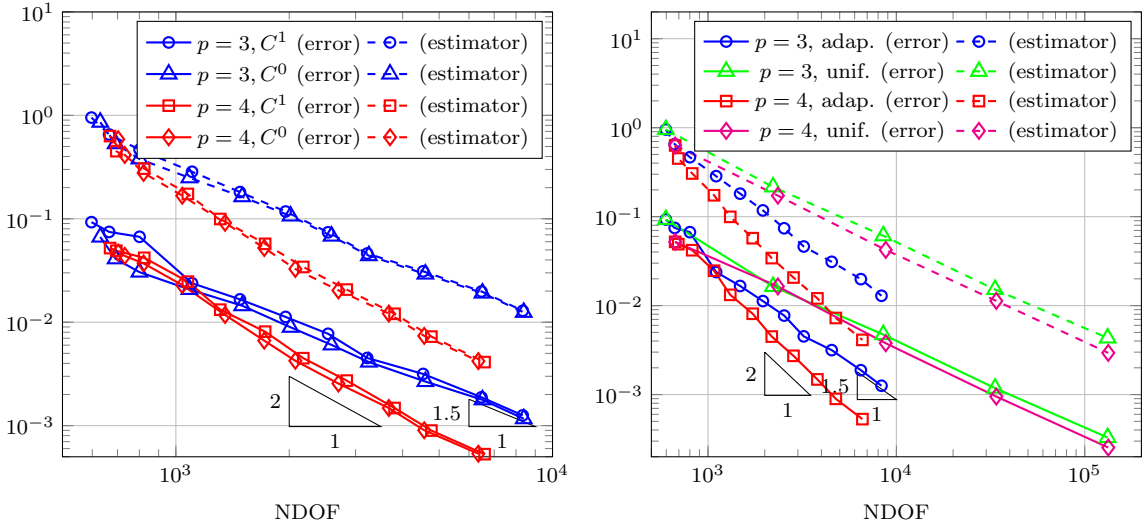
\begin{figure}[ht]
\centering
  \begin{tikzpicture}
  \begin{loglogaxis}[width=.5\textwidth,height=7.5cm,xlabel={NDOF},ylabel={},xmin=5e2,xmax=1e4,ymin=5e-4,ymax=1e1, grid = major, anchor = north, legend columns = 2]
  \addplot[color=blue, mark=o, solid, thick] table [x={ndof}, y={error}, col sep=comma] {ex3_p3_C1.csv};
  \addplot[color=blue, mark=o, dashed, thick, mark options = {solid}] table [x={ndof}, y={estimator}, col sep=comma] {ex3_p3_C1.csv};
  \addplot[color=blue, mark=triangle, solid, thick, mark options = {scale=1.7}] table [x={ndof}, y={error}, col sep=comma] {ex3_p3_C0.csv};
  \addplot[color=blue, mark=triangle, dashed, thick, mark options = {solid, scale=1.7}] table [x={ndof}, y={estimator}, col sep=comma] {ex3_p3_C0.csv};
  \addplot[color=red, mark=square, solid, thick, mark options = {scale=1.}] table [x={ndof}, y={error}, col sep=comma] {ex3_p4_C1.csv};
  \addplot[color=red, mark=square, dashed, thick, mark options = {solid, scale=1.}] table [x={ndof}, y={estimator}, col sep=comma] {ex3_p4_C1.csv};
  \addplot[color=red, mark=diamond, solid, thick, mark options = {scale=1.5}] table [x={ndof}, y={error}, col sep=comma] {ex3_p4_C0.csv};
  \addplot[color=red, mark=diamond, dashed, thick, mark options = {solid, scale=1.5}] table [x={ndof}, y={estimator}, col sep=comma] {ex3_p4_C0.csv};
 \scriptsize \legend{{$p=3, C^1$ (error)}, {(estimator)}, {$p=3, C^0$ (error)}, {(estimator)}, {$p=4, C^1$ (error)}, {(estimator)}, {$p=4, C^0$ (error)}, {(estimator)}}
\convergenceslope{2e3}{3.5e3}{3e-3}{2}{font=\scriptsize,above}{solid}
\convergenceslope{6e3}{9e3}{1.8e-3}{1.5}{font=\scriptsize,above}{solid}
  \end{loglogaxis}
\end{tikzpicture}
  \begin{tikzpicture}
  \begin{loglogaxis}[width=.5\textwidth,height=7.5cm,xlabel={NDOF},ylabel={},xmin=5e2,xmax=2e5,ymin=2e-4,ymax=2e1, grid = major, anchor = north, legend columns = 2]
  \addplot[color=blue, mark=o, solid, thick] table [x={ndof}, y={error}, col sep=comma] {ex3_p3_C1.csv};
  \addplot[color=blue, mark=o, dashed, thick, mark options = {solid}] table [x={ndof}, y={estimator}, col sep=comma] {ex3_p3_C1.csv};
  \addplot[color=green, mark=triangle, solid, thick, mark options = {scale=1.7}] table [x={ndof}, y={error}, col sep=comma] {ex3_p3_C1_glob.csv};
  \addplot[color=green, mark=triangle, dashed, thick, mark options = {solid, scale=1.7}] table [x={ndof}, y={estimator}, col sep=comma] {ex3_p3_C1_glob.csv};
  \addplot[color=red, mark=square, solid, thick, mark options = {scale=1.}] table [x={ndof}, y={error}, col sep=comma] {ex3_p4_C1.csv};
  \addplot[color=red, mark=square, dashed, thick, mark options = {solid, scale=1.}] table [x={ndof}, y={estimator}, col sep=comma] {ex3_p4_C1.csv};
  \addplot[color=magenta, mark=diamond, solid, thick, mark options = {scale=1.5}] table [x={ndof}, y={error}, col sep=comma] {ex3_p4_C1_glob.csv};
  \addplot[color=magenta, mark=diamond, dashed, thick, mark options = {solid, scale=1.5}] table [x={ndof}, y={estimator}, col sep=comma] {ex3_p4_C1_glob.csv};
  \scriptsize \legend{{$p=3$, adap. (error)}, {(estimator)}, {$p=3$, unif. (error)}, {(estimator)}, {$p=4$, adap. (error)}, {(estimator)}, {$p=4$, unif. (error)}, {(estimator)}}
\convergenceslope{2e3}{3.5e3}{3e-3}{2}{font=\scriptsize,above}{solid}
\convergenceslope{6.2e3}{1e4}{1.8e-3}{1.5}{font=\scriptsize,above}{solid}
  \end{loglogaxis}
\end{tikzpicture}
\caption{Error in $H^1$ semi-norm and estimator for Example~\ref{ex:Example3} with ${\bf p}=(3,3)$ and ${\bf p}=(4,4)$, compared with $C^0$ case (left) and with global refinement 
case (right).} \label{fig:conv_ex3}
\end{figure}

\begin{figure}[htb!]\begin{center}
\subfigure[${\bf p}=(3,3)$, $C^0$ functions on the interface: NDOF=$8388$]{\includegraphics[trim=35mm 0mm 30mm 0mm,clip,scale=0.22]{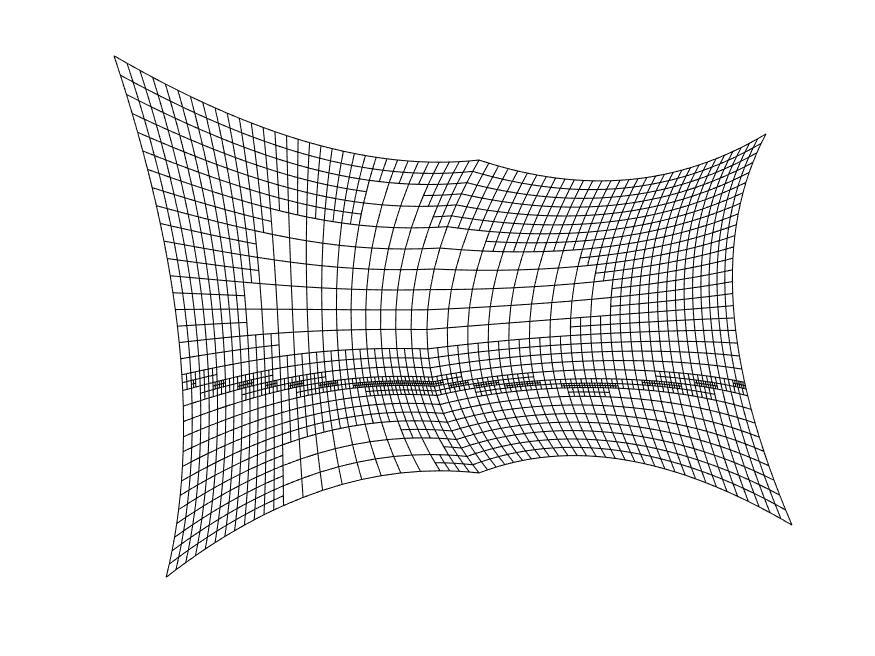}}\hspace{5mm}
\subfigure[${\bf p}=(3,3)$, $C^1$ functions on the interface: NDOF=$8336$]{\includegraphics[trim=35mm 0mm 30mm 0mm,clip,scale=0.22]{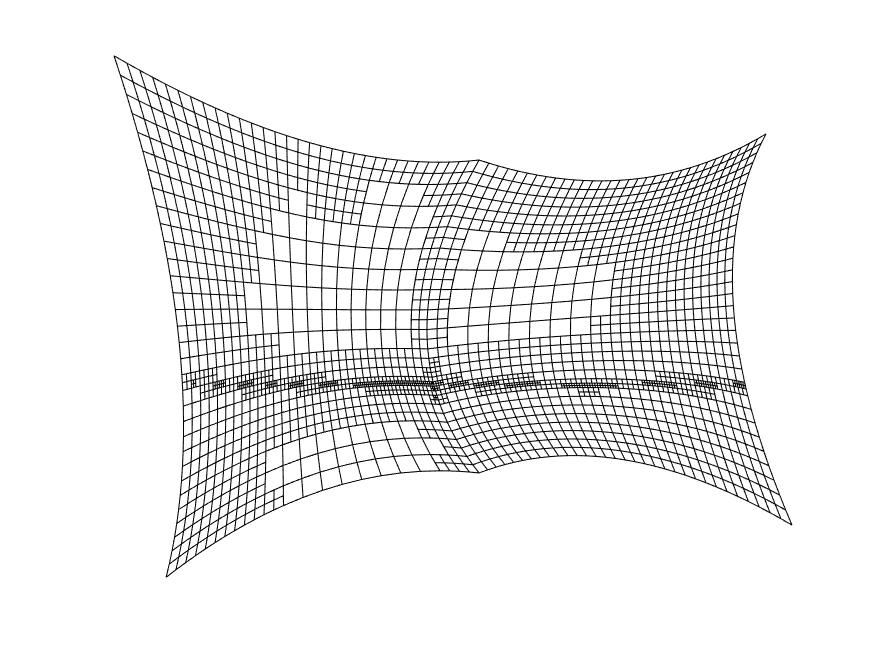}}\\
\subfigure[${\bf p}=(4,4)$, $C^0$ functions on the interface: NDOF=$6356$]{\includegraphics[trim=35mm 0mm 30mm 0mm,clip,scale=0.22]{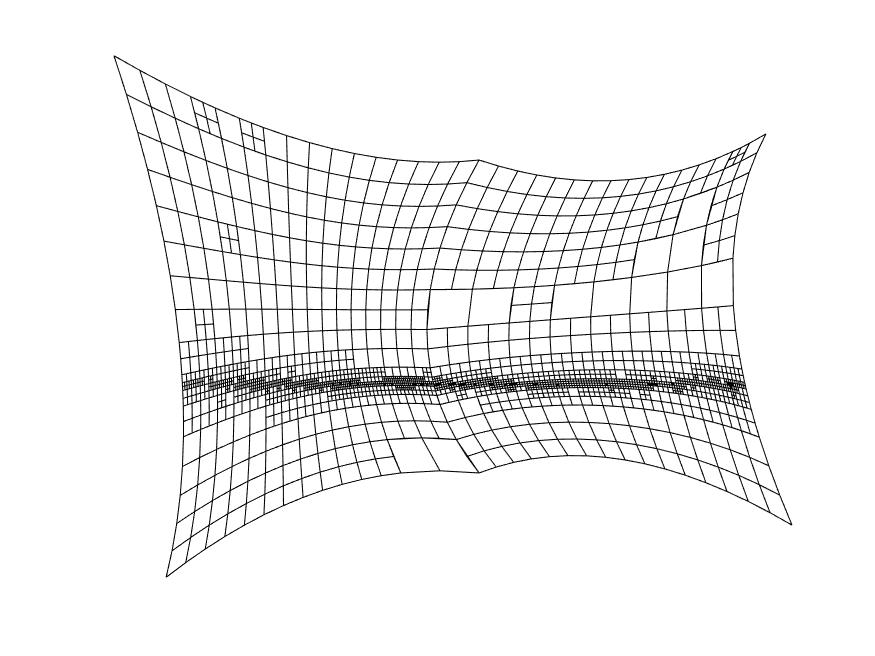}}\hspace{5mm}
\subfigure[${\bf p}=(4,4)$, $C^1$ functions on the interface: NDOF=$6601$]{\includegraphics[trim=35mm 0mm 30mm 0mm,clip,scale=0.22]{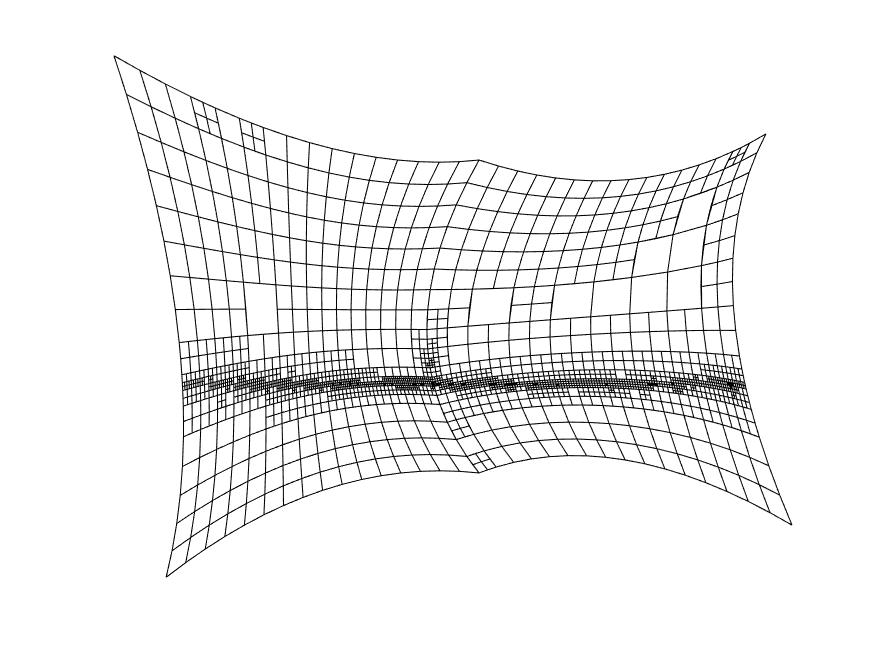}}
\caption{Hierarchical meshes for Example~\ref{ex:Example3}, with ${\bf p}=(3,3)$ and ${\bf p}=(4,4)$.}
 \label{fig:mesh_ex3}
\end{center}
\end{figure}
\end{ex}

\subsection{Bilaplacian problem}
In the last example we consider the solution of the bilaplacian problem, given in strong form by
\[
\left \{
\begin{array}{rl}
\Delta^2 u = f & \text{ in } \Omega, \\
u = g_1 & \text{ on } \partial \Omega, \\
\displaystyle \frac{\partial u}{\partial n} = g_2 & \text{ on } \partial \Omega.
\end{array}
\right.
\]
It is well known that the weak formulation of the problem in direct form requires the trial and test functions to be in $H^2(\Omega)$. For the discretization with a Galerkin method, 
this can be obtained if the discrete basis functions are $C^1$. The solution of the problem with $C^0$ basis functions, instead, requires to use a mixed variational formulation or 
some sort of weak enforcement of the $C^1$ continuity across the interface, like with a Nitsche's method.

\begin{ex} \label{ex:Example4}
For the last numerical test we solve the bilaplacian problem in the L-shaped domain as depicted in Figure~\ref{fig1b}. The right-hand side and the boundary 
conditions are chosen in such a way that the exact solution is given, in polar coordinates $(\rho, \theta)$, by 
\[
u(\rho,\theta)=\rho^{z+1}(C_1\,F_1(\theta) - C_2\,F_2(\theta)),
\]
where value in the exponent is chosen equal to $z= 0.544483736782464$, which is the smallest positive solution of 
\[
\sin (z \omega) + z \sin (\omega) = 0,
\]
with $\omega = 3\pi / 2$ for the L-shaped domain, see \cite[Section~3.4]{Grisvard}. The other terms are given by 
\begin{align*}
&C_1=\frac{1}{z-1} \sin\left(\frac{3(z-1)\pi}{2}\right) - 
\frac{1}{z-1}\sin\left(\frac{3(z+1)\pi}{2}\right),\\
&C_2=\cos\left(\frac{3(z-1)\pi}{2}\right) - \cos\left(\frac{3(z+1)\pi}{2}\right),\\
&F_1(\theta)=\cos((z-1)\theta) - \cos((z+1)\theta),\\
&F_2(\theta)=\frac{1}{z-1}\sin((z-1)\theta) - \frac{1}{z+1}\sin((z+1)\theta).
\end{align*}
The exact solution has a singularity at the reentrant corner, and it is the same kind of singularity that one would encounter for the Stokes problem.

For our numerical test we start with a coarse mesh of $8 \times 8$ elements on each patch. In this case, instead of refining the mesh with an adaptive algorithm we decided to refine 
following a pre-defined strategy: at each refinement step, a region surrounding the reentrant corner, and composed of $4 \times 4$ elements of the finest level, is marked for 
refinement, see Figure~\ref{fig:mesh_ex4}. We remark that the implementation of the adaptive algorithm with a residual-based estimator would require computing fourth order 
derivatives at 
the quadrature points, and several jump terms across the interface, that is beyond the scope of the present work.

In Figure~\ref{fig:conv_ex4} we show the error obtained in $H^2$ semi-norm when computing with $C^1$ hierarchical splines of degrees 3 and 4 and regularity~$r$ equal to 
degree minus two within the single patches, for the local refinement described above, and with $C^1$ isogeometric splines of the same degree and inner regularity~$r$
with global uniform refinement. It is obvious that the hierarchical spaces perform much better, as we obtain a lower error with many less degrees of freedom. 
In this case we do not see a big difference between the results obtained for degrees 3 and 4, but this is caused by the fact that we are refining by hand, and the asymptotic 
regime has not been reached yet. 

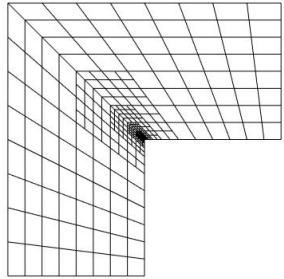
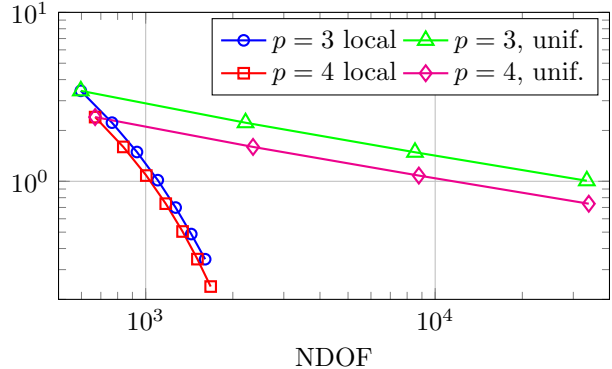
\begin{figure}[h!]\begin{center}
\subfigure[Refinement of the L-shaped domain]{\includegraphics[trim=5mm 0mm 5mm 0mm,clip,scale=0.30]{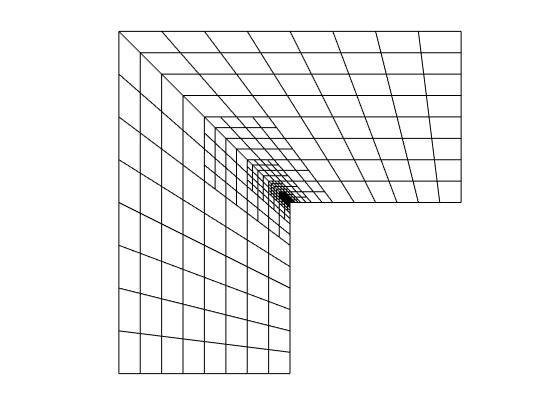} \label{fig:mesh_ex4}}
\subfigure[Error in $H^2$ semi-norm]{
  \begin{tikzpicture}
  \begin{loglogaxis}[width=.55\textwidth,height=5.4cm,xlabel={NDOF},ylabel={},xmin=5e2,xmax=4e4,ymin=2e-1,ymax=1e1, grid = major, anchor = north, legend columns = 2]
  \addplot[color=blue, mark=o, solid, thick] table [x={ndof}, y={error}, col sep=comma] {ex4_p3_C1.csv};
  \addplot[color=green, mark=triangle, solid, thick, mark options = {scale=1.7}] table [x={ndof}, y={error}, col sep=comma] {ex4_p3_C1_glob.csv};
  \addplot[color=red, mark=square, solid, thick, mark options = {scale=1.}] table [x={ndof}, y={error}, col sep=comma] {ex4_p4_C1.csv};
  \addplot[color=magenta, mark=diamond, solid, thick, mark options = {scale=1.5}] table [x={ndof}, y={error}, col sep=comma] {ex4_p4_C1_glob.csv};
 \footnotesize \legend{{$p=3$ local}, {$p=3$, unif.}, {$p=4$ local}, {$p=4$, unif.}}
\convergenceslope{2e3}{3.5e3}{3e-3}{2}{font=\footnotesize,above}{solid}
\convergenceslope{6e3}{9e3}{1.8e-3}{1.5}{font=\footnotesize,above}{solid}
  \end{loglogaxis}
\end{tikzpicture}
\label{fig:conv_ex4}
}
\caption{Hierarchical mesh (a) and comparison of the results obtained by local refinement and $C^1$ space with global refinement (b) on Example~\ref{ex:Example4}.}
 \label{fig:ex4}
\end{center}
\end{figure}
\end{ex}

\section{Conclusions} \label{sec:conclusions}

We presented the construction of $C^1$ hierarchical functions on two-patch geometries and their application in isogeometric analysis. After briefly reviewing the characterization 
of $C^1$ tensor-product isogeometric spaces, we investigated the properties needed to effectively use these spaces as background machinery for the hierarchical spline model. 
In particular, the local linear independence of the one-level basis functions and the nested nature of the considered $C^1$ splines spaces was proved. We also introduced an explicit 
expression of the refinement masks under dyadic refinement, that among other things is useful for the practical implementation of the hierarchical basis functions.
The numerical examples 
show that optimal convergence rates are obtained by the local refinement scheme for second and fourth order problems, even in presence of singular solutions. In future work we plan 
to generalize the construction to the multi-patch domain setting of \cite{KaSaTa17c}, but this will require a different strategy with respect to the approach presented in this work 
since the basis functions of a single level may be locally linearly dependent.

\paragraph*{\bf Acknowledgment}
Cesare Bracco, Carlotta Giannelli and Rafael V\'azquez are members of the INdAM Research group
GNCS. The INdAM support through GNCS and Finanziamenti Premiali SUNRISE is gratefully acknowledged. Rafael V\'azquez has been partially supported by the ERC Advanced Grant “CHANGE”, 
grant number 694515, 2016-2020

\appendix
\normalsize

\section{Geometry of the curved domain} \label{app:geometry}
The geometry in~Fig.\ref{fig1}(a) for the examples in Section~\ref{sec:examples} is generated by following the algorithm in~\cite{KaSaTa17b}. This technique is 
based on solving a quadratic minimization problem with linear side constraints, and constructs from an initial multi-patch geometry~$\widetilde{\f{F}}$ an 
analysis-suitable $G^1$ multi-patch parameterization~$\f{F}$ possessing the same boundary, vertices and first derivatives at the vertices as~$\widetilde{\f{F}}$. 

In our case, the initial geometry~$\widetilde{\f{F}}$ is given by the two patch parameterization consisting of two quadratic B\'ezier patches~$\widetilde{\f{F}}^{(L)}$ and 
$\widetilde{\f{F}}^{(R)}$ (i.e. without any internal knots) with the control points~$\widetilde{\f{c}}_{i,j}^{(S)}$, $S \in \{L,R \}$, specified in Table~\ref{tab:initial-geometry}. 
This parameterization is not analysis-suitable $G^1$.
\begin{table}[ht]
\centering
\setlength{\tabcolsep}{0.5em}
\begin{tabular}{|ccc||ccc|}
\hline
\multicolumn{3}{|c||}{$\widetilde{\f{c}}_{i,j}^{(L)}$} & \multicolumn{3}{c|}{$\widetilde{\f{c}}_{i,j}^{(R)}$} \\
\hline
$(0, 0)$ &    $(-3,1/3)$      & $(-6,-2)$ & $(0,0)$    &  $(13/5, 1)$  & $(6, -1)$ \\
$(-2, 5/2)$ & $(-13/4,53/20)$ & $(-5,2)$  & $(-2, 5/2)$ & $(39/20, 3)$ & $(4, 11/3)$ \\
$(0,6)$ &    $(-3,17/3)$     & $(-7,8)$  & $(0,6)$ &     $(3,5)$ &      $(11/2, 13/2)$ \\
\hline
\end{tabular}
\caption{Control points~$\widetilde{\f{c}}_{i,j}^{(S)}$, $S \in \{L,R\}$, of the initial non-analysis-suitable $G^1$ two-patch parameterization~$\widetilde{\f{F}}$.}
\label{tab:initial-geometry}
\end{table}

Applying the algorithm in~\cite{KaSaTa17b} (by using Mathematica), we construct an analysis-suitable $G^1$ two-patch geometry~$\f{F}$ with bicubic B\'ezier patches~$\f{F}^{(L)}$ and 
$\f{F}^{(R)}$. Their control points~$\f{c}_{i,j}^{(S)}$, $S \in \{L,R \}$, are given in Table~\ref{tab:ASG1}, where for presenting some of their coordinates the notations 
$D = 99170$ and
\[
\begin{array}{ll}
C_1 = 333939/D, & C_2 = 47387036/(22.5 D), \\ 
C_3 = -15800567/(5 D), & C_4 = 242128576/(67.5 D), \\
C_5 = 57452423/(45 D), & C_6 = 81952942/(22.5 D),
\end{array}
\]  
are used.
\begin{table}[ht]
\centering
\setlength{\tabcolsep}{0.5em}
\begin{tabular}{|cccc|}
\hline
\multicolumn{4}{|c|}{$\f{c}_{i,j}^{(L)}$} \\
\hline
$(0, 0)$ &    $(-2,2/9)$      & $(-4,-4/9)$ & $(-6, -2)$  \\
$(-4/3, 5/3)$ & $(-127/50, 44/25)$ & $(-98/25, 37/25)$ & $(-16/3,2/3)$ \\
$(-4/3,11/3)$ & $(C_3,C_4)$ & $(-89/25,189/50)$ & $(-17/3,4)$ \\
$(0,6)$ & $(-2,52/9)$ & $(-13/3,58/9)$ & $(-7,8)$ \\
\hline
\hline 
\multicolumn{4}{|c|}{$\f{c}_{i,j}^{(R)}$} \\ \hline
$(0, 0)$  & $(26/15, 2/3)$ & $(56/15, 1/3)$ & $(6,-1)$ \\
$(-4/3, 5/3)$ & $(C_1,C_2)$ & $(87/25,113/50)$ & $(14/3,19/9)$\\
$(-4/3,11/3)$ & $(C_5,C_6)$ & $(29/10, 4)$ & $(9/2,83/18)$ \\
$(0,6)$ & $(2,16/3)$ & $(23/6,11/2)$ & $(11/2,13/2)$ \\ 
\hline
\end{tabular}
\caption{Control points~$\f{c}_{i,j}^{(S)}$, $S \in \{ L,R\}$, of the resulting analysis-suitable $G^1$ two-patch parameterization~$\f{F}$.}\label{tab:ASG1}
\end{table}

\newpage
\section*{List of symbols}
\footnotesize{
\begin{flushleft}
\begin{tabularx}{\textwidth}{lXlX}
{\normalsize\textbf{Spline space}} &\\[.2cm]
$p$ & Spline degree, $p \geq 3$ &  \\
$r$ & Spline regularity, $1 \leq r \leq p-2$ &  \\
$\UXI{p}{r}$  & Open knot vector & \\
$\tau_i$ & internal breakpoints of knot vector $\UXI{p}{r}$&\\
$\UT$ & Ordered set of internal breakpoints~$\tau_{i}$& \\
$k$ & Number of different internal breakpoints of knot vector $\UXI{p}{r}$& \\ 
$\US{p}{r}$ & Univariate spline space of degree~$p$ and regularity~$r$ on $[0,1]$ over knot vector~$\UXI{p}{r}$ & \\
$\US{p}{r+1}$, $\US{p-1}{r}$  &  Univariate spline spaces of higher regularity and lower degree, respectively, defined from same internal breakpoints as $\US{p}{r}$& \\
$\UN{p}{r}{i}$, $\UN{p}{r+1}{i}$, $\UN{p-1}{r}{i}$ & B-splines of spline spaces~$\US{p}{r}$, $\US{p}{r+1}$ and $\US{p-1}{r}$, respectively  \\
$n$, $n_0$, $n_1$ & Dimensions of spline spaces~$\US{p}{r}$, $\US{p}{r+1}$ and $\US{p-1}{r}$, respectively& \\
$\I$, $\I_0$, $\I_1$ & Index sets of B-splines $\UN{p}{r}{i}$, $\UN{p}{r+1}{i}$ and $\UN{p-1}{r}{i}$, respectively  & \\
$\J{0}{i}$, $\J{1}{i}$ & Index subsets of $\I$ related to B-splines~$\UN{p}{r+1}{i}$ and $\UN{p-1}{r}{i}$, for $i \in \I_0$ and $i \in \I_1$, respectively& \\
$\zeta_m$ & Greville abscissae of spline space~$\US{p}{r}$, $m \in \I$& \\
$\VTN{0}$, $\VTN{1}$, $\VTN{2}$ & Vectors of tensor-product B-splines~$\UN{p}{r}{i}\UN{p}{r}{j}$ &\\[.4cm]
{\normalsize\textbf{Geometry}} &\\[.2cm]
$(S)$ & Upper index referring to specific patch, $S \in \{L,R \}$ & \\
$\Omega^{(S)}$ & Quadrilateral patch & \\
$\Omega$ & Two-patch domain $\Omega = \Omega^{(L)} \cup \Omega^{(R)}$ & \\
$\Gamma$ & Common interface of two-patch domain~$\Omega$&\\
$\f{F}^{(S)}$ & Geometry mapping of patch~$\Omega^{(S)}$ & \\
$\f{F}$ & Two patch geometry $\f{F}=(\f{F}^{(L)},\f{F}^{(R)})$ & \\
$\f{F}_{0}$ & Parameterization of interface~$\Gamma$ \\
$\f{d}$ & Specific transversal vector to $\Gamma$  \\
$\xi_1$, $\xi_2$ & Parameter directions of geometry mappings & \\
$\f{c}_{i,j}^{(S)}$ & Spline control points of geometry mapping~$\f{F}^{(S)}$ & \\
$\alpha^{(S)}$, $\beta^{(S)}$, $\beta$ & Gluing functions of two-patch geometry~$\f{F}$ & \\
$\gamma$ & Scalar function, $\gamma \neq 0$\\[.4cm]
{\normalsize\textbf{$C^1$ isogeometric space}} &\\[.2cm]
$\UV$& Space of $C^1$ isogeometric spline functions on $\Omega$ & \\
$\UW$ & Subspace of $\UV$ & \\
$\PhiB$ & Basis of $\UW$& \\
$\PhiB_{\Omega^{(S)}}$, $\PhiB_{\Gamma_0}$, $\PhiB_{\Gamma_1}$ & Parts of basis~$\PhiB$, 
$\PhiB = \PhiB_{\Omega^{(L)}} \cup \PhiB_{\Omega^{(R)}} \cup \PhiB_{\Gamma_0} \cup \PhiB_{\Gamma_1}$&  \\
$\phiS{S}{i}{j}$ & Basis functions of $\PhiB_{\Omega^{(S)}}$, $i \in \I \setminus \{ 0,1 \}$, $j \in \I$& \\
$\phiT{0}{i}$ & Basis functions of $\PhiB_{\Gamma_0}$, $i \in \I_0$ & \\
$\phiT{1}{i}$ & Basis functions of $\PhiB_{\Gamma_1}$, $i \in \I_1$ & \\
$\hatPhiT{S}{0}$, $\hatPhiT{S}{1}$, $\hatPhiS{S}{S}$ & Vectors of spline functions $\phiT{0}{i} \circ \f{F}^{(S)}$, $\phiT{1}{i} \circ \f{F}^{(S)}$ and 
$\phiS{S}{i}{j}\circ \f{F}^{(S)}$, respectively & \\
$\widehat{B}$, $\widetilde{B}^{(S)}$, $\overline{B}\mbox{}^{(S)}$ &  Transformation matrices & \\
$\coefa{}{i,j}$, $\coefb{(S)}{i,j}$, $\coefc{(S)}{i,j}$ & Entries of matrices $\widehat{B}$, $\widetilde{B}^{(S)}$ and $\overline{B}^{(S)}$, respectively & \\
$B^{(S)}$ & Block matrix assembled by the matrices~$\widehat{B}$, $\widetilde{B}^{(S)}$, $\overline{B}\mbox{}^{(S)}$ and the identity matrix $I_{n(n-2)}$ & \\
& & \\[.4cm]
\end{tabularx}
\begin{tabularx}{\textwidth}{lXlX}
{\normalsize\textbf{Hierarchical space}} &\\[.2cm]
$\ell$& Upper index referring to specific level & \\
$\Lambda_{p}^{r,\ell+1}$, $\Lambda_{p}^{r+1,\ell+1}$, $\Lambda_{p-1}^{r,\ell+1}$ & Refinement matrices for B-splines $\HUN{p}{r}{i}{\ell}$, $\HUN{p}{r+1}{i}{\ell}$ and 
$\HUN{p-1}{r}{i}{\ell}$, respectively& \\
$\lambda_{i,j}^{\ell+1}$& Entries of refinement matrix $\Lambda_{p}^{r,\ell+1}$ & \\
$\Theta_{ij}^{\ell+1}$ & Block matrices of refinement mask $\Lambda_{p}^{r,\ell+1} \otimes \Lambda_{p}^{r,\ell+1}$, $0 \leq i \leq j \leq 2$  & \\
$\UW_H$  & $C^1$ hierarchical isogeometric spline space &\\
${\cal W}$ & Basis of $\UW_H$ &  \\
& & \\[.4cm]
\end{tabularx}
\end{flushleft}
}
Most notations in the paragraphs ``Spline space'' and ``$C^1$ isogeometric space'' can be directly extended to the hierarchical setting by adding the upper 
index~$\ell$ to refer to the considered level.

\end{document}